\documentclass[10pt,a4j]{article}
\usepackage{amsthm}

\newtheorem{Thm}{Theorem}[section]
\newtheorem{Lem}{Lemma}[section]

\newtheorem{Prop}{Proposition}[section]
\newtheorem{Rem}{Remark}[section]

\usepackage{comment}
\usepackage{amsmath, amssymb}
\usepackage{graphicx} 
\usepackage{mathrsfs}
\usepackage{color}
\usepackage{cases}
\usepackage{ascmac}
\usepackage[english]{babel}

\def\XXint#1#2#3{{\setbox0=\hbox{$#1{#2#3}{\int}$}
\vcenter{\hbox{$#2#3$}}\kern-.5\wd0}}

\DeclareMathOperator*{\osc}{\rm osc}

\title{\LARGE \bf Blow up analysis for Boltzmann-Poisson equation in Onsager's theory for point vortices with multi-intensities
}
\author{Takashi Suzuki\footnote{Center for Mathematical Modeling and Data Science, Osaka University, Osaka 560-8531, Japan. (E-mail: suzuki@sigmath.es.osaka-u.ac.jp)} \and Yohei Toyota\footnote{Division of Mathematical Science, Department of Systems Innovation, Graduate School of Engineering Science, Osaka University, Osaka 560-8531, Japan. (E-mail: y-toyota@sigmath.es.osaka-u.ac.jp)} }
\date{}
\begin{document}

\maketitle
\begin{abstract}
In this paper we consider the minimizing sequence for some energy functional of an elliptic equation associated with the mean field limit of the point vortex distribution one-sided Borel probability measure. If such a sequence blows up, we derive some estimate which is related to the behavior of solution near the blow-up point. Moreover, we study the two-intensities case to consider the sufficient condition for this estimate. Our main results are new for the standard mean field equation as well.

%\keywords{blow-up analysis \and  Y. Y. Li type estimate \and Trudinger-Moser inequality \and Onsager's theory \and point vortices \and Boltzman-Poisson equation}
 %\PACS{PACS code1 \and PACS code2 \and more}
 %\subclass{35J61 \and 39B52 \and 76F99}
\end{abstract}

\section{Introduction}
\noindent
Motivated by several mean field equations recently derived in the context of Onsager's statistical mechanics description of turbulence \cite{key040}, we consider the {Boltzmann}-Poisson equation:

\begin{equation} \label{eq100}
-\Delta v= \lambda \int_{I_+} \frac{\alpha e^{\alpha v}}{\int_{\Omega} e^{\alpha v} dx} \mathcal P(d\alpha)\quad {\rm in} \hspace{2mm}\Omega, \quad v=0 \hspace{2mm}{\rm on}\hspace{2mm} \partial\Omega,
\end{equation}
where $\Omega \subset \mathbb R^2$ is  a smooth bounded domain, $v$ denotes the stream function, $\lambda>0$ is a constant related to the inverse temperature and $\mathcal P(d\alpha)$ is a Borel probability measure on $I_+=[0, +1]$ denoting the distribution of the circulations. 
A formal derivation of (\ref{eq100}) is provided in \cite{key016, key052}.\

If $\mathcal P(d\alpha)=\delta_{+1}(d\alpha)$, corresponding to the case where all vortices have the same intensity and orientation, equation (\ref{eq100}) reduces to the Liouville type equation
\begin{equation} \label{eq102}
-\Delta v= \lambda \frac{e^{v}}{\int_{\Omega} e^{v} dx} \quad {\rm in} \hspace{2mm}\Omega, \quad v=0 \hspace{2mm}{\rm on}\hspace{2mm} \partial\Omega.
\end{equation}
\noindent
Equation (\ref{eq102}) is mathematically justified by the minimizing free energy method in the canonical formulation \cite{key04, key028}, and its mathematical analysis has revealed the quantized blow-up mechanism of sequences of solutions, see, e.g., \cite{key00, key032, key034, key060, key062, key064}.\

Especially, the Y. Y. Li type estimate which is the behavior of blow-up solutions for (\ref{eq102}) near the blow-up points has been studied \cite{key022, key030}. Let $\Omega$ be a unit ball and $(\lambda_k, v_k)$ satisfy (\ref{eq102}) without boundary condition and 

\begin{equation} \label{eq104}
\lambda_k \to \lambda_0 \geq 0, \quad \| v_k \|_{\infty}=v_k(x_k) \to +\infty, \quad x_k \to 0 \in \Omega
\end{equation}
as $k \to +\infty$ where $x_k$ is the maximizer of $v_k$ and $0$ is the only blow-up point of $v_k$. Then the following result holds:

\begin{Thm} {\rm(\cite{key022}, Theorem 0.3)} \label{th100}
Under the blow-up case (\ref{eq104}), {suppose that} there exists a constant $C>0$ such that
\begin{equation} \label{eq106}
\max_{\partial \Omega} v_k -\min_{\partial \Omega} v_k \leq C.
\end{equation}
Then it holds that
\begin{equation} \label{eq108}
v_k(x)-v_k(x_k)=-2\log\Bigg(1+\frac{\lambda_k}{8} \frac{e^{v_k(x_k)}}{\int_{\Omega} e^{v_k}}|x-x_k|^2\Bigg)+O(1)
\end{equation}
as $k \to \infty$ uniformly $x \in B_r(0)$ with some $0<r<1$.
\end{Thm}

\begin{Rem} \label{rem100}
We can understand (\ref{eq106}) as boundary condition in Theorem \ref{th100} and there are no need to suppose the zero Dirichlet boundary condition for Theorem \ref{th100}.  
\end{Rem}

Y. Y. Li type estimate of (\ref{eq108}) is valid for the computation of the Leray-Schauder degree for (\ref{eq102}), asymptotic non-degeneracy of multi-point blowup solutions to {the Liouville Gel'fand problem} and the Trudinger-Moser inequality with the extremal case, see \cite{key022, key036, key058}.\

It is known that there are two proofs for Theorem \ref{th100}. The first one which is the original way of Y. Y. Li, is the combination with some conformal transformation and the moving plane argument \cite{key022}. The other one is the argument of C. S. Lin \cite{key030}. In \cite{key030}, we can control the mass of bubble in the quantized blow-up argument thanks to the boundary condition (\ref{eq106}). By such a information of mass and a result of \cite{key08}, we obtain the {\it mass identity} which is described precisely later, and this identity plays an essential role in the proof of Theorem \ref{th100}. \

Comparing with the case $\mathcal P(d\alpha)=\delta_{+1}(d\alpha)$, however, there are no works of describing the Y. Y. Li type estimate for mean field equation in the multi-intensities case. Our aim in this paper is to derive the variant of Y. Y. Li type estimate in the multi-intensities case (\ref{eq100}). To {achieve this}, we shall employ the argument of \cite{key058}.\
Here, we introduce some notations and assumptions to describe our results. \\
Setting
\begin{equation*}
J_{\lambda}(v)=\frac{1}{2}\|\nabla v\|_2^2-\lambda\int_{I_+}\log\Big(\int_{\Omega} e^{\alpha v} dx \Big)\mathcal P(d\alpha), \quad v \in H_0^1(\Omega),
\end{equation*}
then equation (\ref{eq100}) is the Euler-Lagrange equation of this functional.\\
The extremal value of $\lambda$ for $\inf_{v \in H_0^1(\Omega)} J_{\lambda} (v)>-\infty$ is defined by
\begin{align}
\overline\lambda:&=\sup\Big\{\lambda>0 \mid \inf_{v\in H_0^1(\Omega)} J_{\lambda}(v)>-\infty\Big\}. \label{eq124'}
\end{align}
This extremal value is actually given by \cite{key044}, that is,
\begin{equation}\label{eq126}
\overline \lambda = \inf \Biggr\{ \frac{8\pi \mathcal P(K)}{\Big(\int_{K} \alpha \mathcal P(d\alpha)\Big)^2} \mid K \subset  \rm supp \mathcal P \Biggl\},
\end{equation}
where supp $\mathcal P$=$\{\alpha \in I_+ \mid \mathcal P(N)>0$ for any open neighborhood $N$ of $\alpha$ $\}$.\\
Then it holds that
\begin{align*}
\lambda<\overline \lambda \quad &\Rightarrow \inf_{v \in H_0^1(\Omega)} J_{\lambda} (v)>-\infty,\\
\lambda>\overline \lambda \quad &\Rightarrow \inf_{v \in H_0^1(\Omega)} J_{\lambda} (v)=-\infty.
\end{align*}
Therefore, given $\lambda_k \uparrow \overline \lambda$, we have a minimizer $v_k \in H_0^1(\Omega)$ of ${J_{\lambda_k}}$, and $(\lambda_k, v_k)$ satisfy (\ref{eq100}).  For the solution sequence to (\ref{eq100}), the following Brezis-Merle type blow-up alternatives holds \cite{key034, key046, key049}:

\begin{Prop}  \label{prop100}
Let $(\lambda_k, v_k)$ be a solution sequence of (\ref{eq100}) with {$\lambda_k >0$ and $\lambda_k \to \lambda_0$}. Assume that
\begin{equation} \label{eq112}
S \cap\partial\Omega=\emptyset
\end{equation}
holds, where $S=\{x_0 \in \overline\Omega \mid$ there exists $x_k \in \Omega$ such that ${x_k \to x_0}$ and $v_k(x_k) \to \infty\}$. Then, passing to a subsequence, we have the following alternatives.\\
(I) Compactness: $\limsup_{k \to \infty} \| v_k \|_{\infty}<+\infty$, that is, $S=\emptyset$.\

Then, there exists $v \in H_0^1(\Omega)$ such that $v_k \to v$ in $H_0^1(\Omega)$ and $v$ is a solution of (\ref{eq100}).\\
(II) Concentration: $\limsup_{k \to \infty} \| v_k \|_{\infty}=+\infty$, that is, $S \not= \emptyset$.\

Then, $S$ is finite and there exists $0\leq s(x) \in L^1(\Omega)\cap L_{loc}^{\infty}(\Omega\setminus S)$ such that
\begin{equation} \label{eq114}
\mu_k(dx) \equiv \lambda_k \int_{I_+} \frac{\alpha e^{\alpha v}}{\int_{\Omega} e^{\alpha v}dx} \mathcal P(d\alpha) dx \overset{*}{\rightharpoonup} s(x)dx+\sum_{x_0 \in S}m (x_0)\delta_{x_0}(dx) \quad {\rm in} \hspace{2mm} \mathcal M(\overline\Omega),
\end{equation}
with $m(x_0) \geq 4\pi$ where $\delta_{x_0}$ denotes the Dirac measure centered at $x_0$ and $\mathcal M(\overline\Omega)$ is the space of measures identified with the dual space of $C_0(\Omega)$.
\end{Prop}

\begin{Rem} \label{rem102}
If we apply Proposition \ref{prop100} to the solution $(\lambda_k, v_k)$ of (\ref{eq102}), then it is known that we get the more detail of the blow-up information. For example,  $s(x) \equiv 0$ in $\Omega$, which we call residual vanishing and $m(x_0) \in 8\pi \mathbb N $ for $x_0 \in S$ \cite{key00, key032}.
\end{Rem}

\noindent
{Since the minimizing sequence $(\lambda_k, v_k)$ satisfies (\ref{eq100}), we can apply Proposition \ref{prop100} to it if we get the condition (\ref{eq112}). In general, thanks to a result in \cite{key018}, p.223, (\ref{eq112}) follows for solution sequence to (\ref{eq100}). It is enough to check the following statement:} 

\begin{Lem} \label{lem100}
Let $(\lambda_k, v_k)$ be a solution sequence to (\ref{eq100}) with {$\lambda_k >0$ and $\lambda_k \to \lambda_0$}. There exists a tubular neighborhood $\Omega_{\delta}$ of $\partial\Omega$ and a constant $C>0$ such that $\| {v_k} \|_{L^{\infty}(\Omega_{\delta})} \leq C$ for any $k \in \mathbb N$.
\end{Lem}

\noindent
{The proof of Lemma \ref{lem100} is almost the same as in \cite{key050}, Lemma 2.5. Therefore, we have
\begin{equation} \label{eq116}
S \cap\partial\Omega=\emptyset, \quad \# S<\infty 
\end{equation}
for minimizing sequence $(\lambda_k, v_k)$. In the following, we consider the minimizing sequence $(\lambda_k, v_k)$ for ${J_{\lambda_k}}$ in the Concentration case, that is,
\begin{equation} \label{eq110}
\lambda_k \to \overline\lambda, \quad \| v_k \|_{\infty}=v_k(x_k) \to \infty \quad as \hspace{2mm} k \to \infty
\end{equation}
where $x_k$ is the maximizer of $v_k$. Indeed, if $\mathcal P$ is the one-intensity or two-intensity case and $\Omega$ is a ball then (\ref{eq110}) is justified \cite{key04, key048}.} By (\ref{eq116}), up to a subsequence, $x_k \to x_0 \in \Omega$.\

Next, we define
\begin{align*}
w_{k, \alpha}(x)&:=\alpha v_k(x+x_k)-\log \int_{\Omega} e^{\alpha v_k}, \quad k\in \mathbb N, \quad \alpha \in I_{+}\setminus\{0\} \hspace{2mm} {\rm and}\\
w_k(x)&:=w_{k, 1}(x).
\end{align*}
Then we have 
\begin{equation} \label{eq118}
-\Delta w_k=\lambda_k \int_{I_+} \alpha e^{w_{k, \alpha}}\mathcal P(d\alpha) \quad {\rm in} \hspace{2mm} \Omega, \quad \int_{\Omega} e^{w_{k, \alpha}}=1,
\end{equation}
and we shall show that for $\alpha \in I_+ \setminus\{0\}$,
\begin{equation*} 
w_k(0) \geq w_{k, \alpha}(0) \to +\infty, \quad k \to \infty.
\end{equation*}
Furthermore, setting
\begin{equation*}
\tilde{w}_{k, \alpha}(x):=w_{k, \alpha}(\sigma_k x)+2\log \sigma_k, \quad \sigma_k=e^{-w_k(0)/2}\to 0, \quad \tilde{w}_k:=\tilde{w}_{k, 1}
\end{equation*}
then, we obtain 
\begin{equation*}
-\Delta \tilde{w}_k =\tilde{f}_k, \quad \tilde{w}_k(x) \leq \tilde{w}_k(0)=0 \quad {\rm in} \hspace{2mm} B_{R_0{\sigma_k}^{-1}},
\end{equation*}
where $\tilde{f}_k:= \lambda_k \int_{I_+} \alpha e^{\tilde{w}_{k, \beta}}\mathcal P(d\alpha)$, $4R_0=dist(x_0, \partial\Omega)$. By {elliptic regularity arguments}, we can show that there exists $\tilde{w}$, $\tilde{f}$ $\in C^2(\mathbb R^2)$ such that
\begin{equation*}
\tilde{w}_k \to \tilde{w}, \quad \tilde{f}_k \to \tilde{f} \quad {\rm in} \hspace{2mm} C^2_{loc}(\mathbb R^2),
\end{equation*}
and 
\begin{equation*} 
-\Delta \tilde{w}=\tilde{f}\not\equiv0, \quad \tilde{w} \leq \tilde{w}(0)=0, \quad 0\leq \tilde{f} \leq {\overline\lambda}\int_{I_+} \alpha \mathcal P(d\alpha)\quad {\rm in} \hspace{2mm} \mathbb R^2,
\end{equation*}
\begin{equation*}
\int_{\mathbb R^2} e^{\tilde{w}} \leq1, \quad \int_{\mathbb R^2} \tilde{f} \leq {\overline\lambda}\int_{I_+} \alpha \mathcal P(d\alpha).
\end{equation*}
Then we assume that
\begin{equation} \label{eq120}
\beta_0:=\int_{\mathbb R^2} \tilde{f} dx= m(x_0),
\end{equation}
where $m(x_0)$ {is} as in (\ref{eq114}).
\begin{Rem} \label{rem104}
Since (\ref{eq120}) means that the total mass of scaling limit coincides with the local mass of bubble, we call (\ref{eq120}) mass identity. Indeed, in Theorem \ref{th100}, the both sides of (\ref{eq120}) coincides with $8\pi$ by a result of \cite{key08, key030}. 
\end{Rem}
\noindent
In addition to (\ref{eq110}) and (\ref{eq120}), we also assume 
\begin{equation} \label{eq121'}
{\alpha_{min}>0\quad and} \quad \mathcal P(\{\alpha_{min} \})>0,
\end{equation}
where supp $\mathcal P$=$\{\alpha \in I_+ \mid \mathcal P(N)>0$ for any open neighborhood $N$ of $\alpha \}$ and $\alpha_{min}=\inf_{\alpha \in supp \mathcal P} \alpha$.
Then the variant of Y. Y. Li type estimate holds:

\begin{Thm} \label{th102}
Suppose (\ref{eq110}), (\ref{eq120}), (\ref{eq121'}) and $s(x) \equiv 0$ as in (\ref{eq114}) then it holds that
\begin{equation} \label{eq122}
v_k(x) - v_k(x_k)=-\Big(\frac{\beta_0}{2\pi}+o(1)\Big)\log\Big(1+\Big(\frac{e^{v_k(x_k)}}{\int_{\Omega} e^{v_k}}\Big)^{\frac{1}{2}}|x-x_k|\Big)+O(1)
\end{equation}
as $k \to \infty$ uniformly in $B_{R_0/2}(x_0)$ where $\beta_0=\int_{\mathbb R^2} \tilde{f} (x) dx$.
\end{Thm}

\begin{Rem}\label{rem105}
$s(x) \equiv 0$ which we call residual vanishing,  occurs under the suitable {assumptions on} $\mathcal P$. Indeed, if $\alpha_{min} > 1/2$ then the residual vanishing occurs to the $(\lambda_k, v_k)$ in (\ref{eq110}) (\cite{key056}, Theorem 3). Moreover, if the residual vanishing occurs to the above $(\lambda_k, v_k)$ then it follows that
\begin{equation*}
\# S=1, \quad \overline \lambda=\frac{8\pi}{\Big(\int_{I_+} \alpha \mathcal P(d\alpha) \Big)^2},
\end{equation*}
see \cite{key056}, Lemma 3.
\end{Rem}

\begin{Rem} \label{rem106}
{The estimate (\ref{eq122}) is weaker than (\ref{eq108}).} Indeed, if $\mathcal P(d\alpha)=\delta_1(d\alpha)$ then $\beta_0=8\pi$ by Chen-Li \cite{key08} and (\ref{eq122}) does not correspond to (\ref{eq108}). However, by a direct calculation, (\ref{eq122}) leads to (\ref{eq108}) with the case $\mathcal P(d\alpha)=\delta_1(d\alpha)$ {in the meaning of the log function term}. Indeed, suppose $(\lambda_k, v_k)$ satisfy (\ref{eq110}) then it holds that\\
(i)
\begin{equation*} 
\Bigg(1+\Big(\frac{e^{v_k(x_k)}}{\int_{\Omega} e^{v_k}}\Big)^{\frac{1}{2}}|x-x_k|\Bigg)^2=\Bigg(1+\frac{e^{v_k(x_k)}}{\int_{\Omega} e^{v_k}}|x-x_k|^2\Bigg)(1+o(1)) \quad as \quad k\to+\infty,
\end{equation*}
(ii)
\begin{equation*} 
\log\Bigg(1+\frac{e^{v_k(x_k)}}{\int_{\Omega} e^{v_k}}|x-x_k|^2\Bigg)=\log \Bigg(1+{\lambda_k}\frac{e^{v_k(x_k)}}{\int_{\Omega} e^{v_k}}|x-x_k|^2\Bigg)+O(1) \quad as \quad k\to+\infty,
\end{equation*}
(iii)
\begin{equation*}
1+\frac{\lambda_k}{8}\frac{e^{v_k(x_k)}}{\int_{\Omega} e^{v_k}}|x-x_k|^2=\Bigg(1+{\lambda_k}\frac{e^{v_k(x_k)}}{\int_{\Omega} e^{v_k}}|x-x_k|^2\Bigg) \cdot O(1) \quad as \quad k\to+\infty,
\end{equation*}
uniformly $B_{R_0/2}(x_0)$ as in Theorem \ref{th102}.
Applying (i), (ii) and (iii) to (\ref{eq122}), we have {the form of} (\ref{eq108}) as $k \to \infty$.
\end{Rem}

For the sufficient conditions of Theorem \ref{th102}, we consider the following identity:
\begin{equation} \label{eq124}
\int_{\mathbb R^2} \tilde{f} dx=\overline \lambda \int_{I_+} \alpha \mathcal P(d\alpha).
\end{equation}
The above identity {implies} the following Proposition.

\begin{Prop} \label{prop102}
Under the {assumption of} $(\lambda_k, v_k)$ in (\ref{eq110}), (\ref{eq124}) holds if and only if the residual vanishing occurs and mass identity (\ref{eq120}) holds.
\end{Prop}

Lastly, we derive the identity (\ref{eq124}) in the minimizing problem with $\mathcal P(d\alpha)$ two-intensities, that is,
\begin{equation} \label{eq136}
\mathcal P(d\alpha)=\tau \delta_1(d\alpha)+(1-\tau)\delta_{\gamma}(d\alpha),
\end{equation}
where $\tau, \gamma \in (0, 1)$ and note that
\begin{equation} \label{eq137}
\overline \lambda
=\begin{cases}
   \frac{8\pi}{\tau}, &\gamma \leq \frac{\sqrt{\tau}}{1+\sqrt{\tau}} \\
   \frac{8\pi}{(\tau+(1-\tau)\gamma)^2},  &\gamma >\frac{\sqrt{\tau}}{1+\sqrt{\tau}}.
  \end{cases}
\end{equation}
The following statements hold under the assumption of $(\lambda_k, v_k)$ in (\ref{eq110}):
\begin{Thm} \label{th106}
(i) If $\mathcal P(d\alpha)$ is as in (\ref{eq136}) and $\gamma \in (\sqrt{\tau}/(1+\sqrt{\tau}), 1)$ then the  identity (\ref{eq124}) holds and the Y. Y. Li type estimate as in (\ref{eq122}) also holds.\\
(ii) If $\mathcal P(d\alpha)$ is as in (\ref{eq136}) and $\gamma \in (0, \sqrt{\tau}/(1+\sqrt{\tau}))$ then the identity (\ref{eq124}) does not hold.
\end{Thm}
{
\begin{Rem} \label{rem108}
Proposition \ref{prop100} and Lemma \ref{lem100} follow for the general solution sequence $(\lambda_k, v_k)$, while our main results Theorem \ref{th102}-\ref{th106} describe just for minimizing sequence $(\lambda_k, v_k)$. In particular, to obtain the estimate (\ref{eq122}) for the general blow-up solution sequence, we have to assume the identity like (\ref{eq124}). In such a case, however, we do not know this identity holds or not. For the proof of Theorem \ref{th106}, we need the property of $\overline \lambda$. This detail shall be mentioned  as Remark \ref{rem400}-\ref{rem402} in Section 4. 
\end{Rem}
}

Our paper is composed of four sections and Appendix. First, we shall discuss the blow-up argument  for general $\mathcal P$ as Preliminary in Section 2. Next, we show Theorem \ref{th102} in Section 3. Lastly, we prove Theorem \ref{th106} and Proposition \ref{prop102} in Section 4. An auxiliary lemma of Section 2 in Appendix.

\section{Preliminary}
In this section, we discuss the blow-up argument for $(\lambda_k, v_k)$ in (\ref{eq110}) without residual vanishing.

\begin{Lem} \label{lem201}
For $\alpha \in I_+$, we have
\begin{equation}\label{eq200}
\frac{d}{d\alpha}w_{k, \alpha}(0) \geq 0,
\end{equation}
where ${w}_{k, \alpha}(x)=\alpha v_k(x+x_k)-\log\int_{\Omega}e^{\alpha v_k}$.
\end{Lem}

\begin{proof}
For $k$ and $\alpha \in I_+$, we have
\begin{equation*}
\frac{d}{d\alpha}w_{k, \alpha}(0) =v_k(x_k)-\frac{\int_{\Omega} v_k e^{\alpha v_k}}{\int_{\Omega} e^{\alpha v_k}}\geq v_k(x_k)\Big(1-\frac{\int_{\Omega} e^{\alpha v_k}}{\int_{\Omega} e^{\alpha v_k}} \Big)=0,
\end{equation*}
recalling that $x_k$ is the maximizer of $v_k$. 
\end{proof}
Henceforth, we put
\begin{equation*}
w_k(x)=w_{k, 1}(x).
\end{equation*}
It follows from (\ref{eq200}) that
\begin{equation} \label{eq203}
w_{k, 1}(0)=\max_{\alpha \in I_+}w_{k, \alpha}(0).
\end{equation}
The following Lemma is the starting point of our {blow-up} analysis.
\begin{Lem} \label{lem202}
For every $\alpha \in I_+ \setminus\{0\}$, it holds that
\begin{equation*}
w_{k, \alpha}(0)=\max_{\Omega} w_{k, \alpha}\to +\infty \quad as \hspace{2mm}k\to\infty.
\end{equation*}
\end{Lem}
\begin{proof}
Since $e^{w_{k, \alpha}(0)}=e^{\alpha v_k}/\int_{\Omega} e^{\alpha v_k} \geq |\Omega|^{\alpha-1}e^{\alpha w_{k, 1}(0)}$ for $\alpha \in I_+ \setminus\{0\}$, it suffices to show that $w_k(0)=w_{k, 1}(0) \to +\infty$ as $k \to +\infty$. Suppose $w_k(0)=O(1)$ as $k \to +\infty$, from (\ref{eq200}) we have $w_{k, \alpha}(0)=O(1)$ as $k \to +\infty$ for all $\alpha \in I_+ \setminus\{0\}$. Therefore the right-hand side on the equation (\ref{eq100}) is uniformly bounded. This contradicts to (\ref{eq110}) from {elliptic regularity arguments}. 
\end{proof}

Putting
\begin{equation} \label{eq204}
\tilde{w}_{k, \alpha}(x):=w_{k, \alpha}(\sigma_k x)+2\log \sigma_k, \quad \sigma_k=e^{-w_k(0)/2}\to 0, \quad \tilde{w}_k:=\tilde{w}_{k, 1}.
\end{equation}
Then, we have 
\begin{equation}  \label{eq205}
-\Delta \tilde{w}_k =\tilde{f}_k, \quad \tilde{w}_k(x) \leq \tilde{w}_k(0)=0 \quad {\rm in} \hspace{2mm} B_{R_0    \sigma_k^{-1} (0)},
\end{equation}
\begin{equation} \label{eq207}
\int_{B_{R_0 \sigma_k^{-1}}(0)} e^{\tilde{w}_{k, \beta}} \leq 1, \quad \int_{B_{R_0 \sigma_k^{-1}}(0)} \tilde{f}_k \leq \lambda_k \int_{I_+} \beta\mathcal P(d\beta),
\end{equation}
where
\begin{align} \label{eq209}
\tilde{f}_k:= \lambda_k \int_{I_+} \beta e^{\tilde{w}_{k, \beta}}\mathcal P(d\alpha), \quad 4R_0 ={dist(x_0, \partial\Omega)}.
\end{align}
We shall use a fundamental fact of which proof is provided in Appendix.
\begin{Lem} \label{lem203}
Given $f \in L^1 \cap L^{\infty}(\mathbb R^2)$, let
\begin{equation*}
z(x)=\frac{1}{2\pi}\int_{\mathbb R^2}f(y)\log\frac{|x-y|}{1+|y|}dy.
\end{equation*}
Then, it holds that
\begin{equation*}
\lim_{|x|\to+\infty} \frac{z(x)}{\log|x|} \equiv \frac{1}{2\pi}\int_{\mathbb R^2} f.
\end{equation*}
\end{Lem}
The following lemma is also classical (see \cite{key042} p. 130).
\begin{Lem} \label{lem204}
If $\phi=\phi(x)$ is a harmonic function on the whole space $\mathbb R^2$ such that
\begin{equation*}
\phi(x) \leq C_1(1+\log|x|), \quad x \in \mathbb R^2 \setminus B_1
\end{equation*}
then it is a constant function.
\end{Lem}

\begin{Prop} \label{pro201}
There exists $\tilde{w}$, $\tilde{f}$ $\in C^2(\mathbb R^2)$ such that
\begin{equation} \label{eq210}
\tilde{w}_k \to \tilde{w}, \quad \tilde{f}_k \to \tilde{f} \quad {\rm in} \hspace{2mm} C^2_{loc}(\mathbb R^2),
\end{equation}
and 
\begin{equation} \label{eq212}
-\Delta \tilde{w}=\tilde{f}\not\equiv0, \quad \tilde{w} \leq \tilde{w}(0)=0, \quad 0\leq \tilde{f} \leq {\overline\lambda}\int_{I_+} \beta \mathcal P(d\alpha)\quad {\rm in} \hspace{2mm} \mathbb R^N,
\end{equation}
\begin{equation*}
\int_{\mathbb R^2} e^{\tilde{w}} \leq1, \quad \int_{\mathbb R^2} \tilde{f} \leq {\overline\lambda} \int_{I_+} \beta\mathcal P(d\beta).
\end{equation*}
In addition, for $x \in \mathbb R^2$,
\begin{equation} \label{eq214}
\tilde{w}(x) \geq -\frac{\beta_0}{2\pi}\log(|x|+1)+\frac{1}{2\pi}\int_{\mathbb R^2} \tilde{f}(y) \log\frac{|y|}{1+|y|}
\end{equation}
where $\beta_0=\int_{\mathbb R^2} \tilde{f}(y)dy$.
\end{Prop}

\begin{proof}
We have
\begin{equation} \label{eq216}
\tilde{w}_{k, \beta}(x) =\beta\tilde{w}_k(x)+(w_{k, \beta}(0)-w_k(0))
\end{equation}
for any $\beta \in I_+ \setminus\{0\}$, and also
\begin{equation} \label{eq218}
\tilde{w}_k \leq \tilde{w}_k(0)=0, \quad w_{k, \beta}(0) \leq w_k(0), \quad \beta \in I_+\setminus\{0\}
\end{equation}
by (\ref{eq200}). Hence $\tilde{f}_k=\tilde{f}_k(x)$ satisfies
\begin{equation} \label{eq220}
0 \leq \tilde{f}_k(x) \leq \lambda_k\int_{I_+} \beta\mathcal P(d\beta) \quad {\rm in} \hspace{2mm} B_{R_0 \sigma_k^{-1}}(0).
\end{equation}
Fix $L>0$ and decompose $\tilde{w}_k$, $k \gg 1$, as $\tilde{w}_k=\tilde{w}_{1, k}+\tilde{w}_{2, k}$ where $\tilde{w}_{j, k}$, $j=1, 2$, are the solutions to
\begin{align*}
-\Delta \tilde{w}_{1, k}&=\tilde{f}_k \quad {\rm in} \hspace{2mm} B_L, \quad \tilde{w}_{1, k}=0 \quad \hspace{2mm}{\rm on} \hspace{2.5mm} \partial B_L,\\
-\Delta \tilde{w}_{2, k}&=0 \quad \hspace{1.5mm}{\rm in} \hspace{2mm} B_L, \quad \tilde{w}_{2, k}=\tilde{w}_{k} \quad {\rm on} \hspace{2mm} \partial B_L.
\end{align*}

First, by (\ref{eq220}) and {elliptic regularity arguments}, there exists $C_{1, L}>0$ such that
\begin{equation*}
0 \leq \tilde{w}_{1, k} \leq C_{1, L} \quad {\rm on} \hspace{2mm} \overline B_L.
\end{equation*}
Next it follows from $\tilde{w}_k \leq 0$ that
\begin{equation*}
\tilde{w}_{2, k} \leq 0 \quad {\rm on} \hspace{2mm} \overline B_L.
\end{equation*}
Hence $\tilde{w}_{2, k}=\tilde{w}_{2, k}(x)$ is a negative harmonic function in $B_L$. Then the Harnack inequality yield $C_{2, L}>0$ such that
\begin{equation*}
\tilde{w}_{2, k} \geq -C_{2, L} \quad {\rm in} \hspace{2mm} \overline B_{L/2}.
\end{equation*}
We thus end up with 
\begin{equation} \label{eq222}
-C_{2, L} \leq \tilde{w}_k \leq \tilde{w}_k(0)=0 \quad {\rm in} \hspace{2mm} B_{L/2},
\end{equation}
and then {standard elliptic regularity arguments assure} the limit (\ref{eq210}) and (\ref{eq212}) thanks to (\ref{eq220}) and (\ref{eq222}).\

If $\tilde{f} \equiv 0$ then
\begin{equation*}
-\Delta \tilde{w}=0, \quad \tilde{w} \leq \tilde{w}(0) =0 \quad {\rm in} \hspace{2mm} \mathbb R^2, \quad \int_{\mathbb R^2} e^{\tilde{w}} \leq 1,
\end{equation*}
which is impossible by the Liouville theorem, and hence $\tilde{f} \not\equiv 0$.\

Since $\tilde{f} \in L^1 \cap L^{\infty}(\mathbb R^2)$, the function
\begin{equation}\label{eq224}
\tilde{z}(x)=\frac{1}{2\pi}\int_{\mathbb R^2} \tilde{f}(y)\log\frac{|x-y|}{1+|y|}dy
\end{equation}
is well-defined, and satisfies
\begin{equation}\label{eq226}
\frac{\tilde{z}(x)}{\log |x|} \to \frac{\beta_0}{2\pi}=\frac{1}{2\pi}\int_{\mathbb R^2} \tilde{f} \quad as \hspace{2mm} |x| \to \infty
\end{equation}
by Lemma \ref{lem203}. Also (\ref{eq226}) implies
\begin{equation*}
-\Delta \tilde{w}=\tilde{f}, \quad -\Delta \tilde{z}=-\tilde{f}, \quad \tilde{w} \leq \tilde{w}(0)=0 \quad {\rm in} \hspace{2mm} \mathbb R^2,
\end{equation*}
\begin{equation*}
\tilde{z}(x) \leq \Big(\frac{\beta_0}{2\pi}+1\Big)\log |x|, \quad x \in \mathbb R^2 \setminus B_r
\end{equation*}
for some $r>0$ by (\ref{eq226}). Hence we obtain $\tilde{u} \equiv \tilde{w}+\tilde{z} \equiv$ constant by Lemma \ref{lem204}. Since $\tilde{w}(0)=0$ it holds that
\begin{equation} \label{eq228}
\tilde{w}(x)=-\tilde{z}+\tilde{z}(0).
\end{equation}
Now we note
\begin{align*}
\tilde{z}(x) &\leq\frac{1}{2\pi} \int_{\mathbb R^2} \tilde{f} \log\frac{|x|+|y|}{1+|y|}dy \\
&\leq \log(1+|x|)\cdot \frac{1}{2\pi} \int_{\mathbb R^2} \tilde{f}=\frac{\beta_0}{2\pi}\log(1+|x|)
\end{align*}
by $\tilde{f} \geq 0$. Hence, $\tilde{w}(x) \geq -\frac{\beta_0}{2\pi}\log(1+|x|)+\tilde{z}(0)$, and the proof is completed. 
\end{proof}

Next we focus on the quantity $\beta_0=\int_{\mathbb R^2} \tilde{f}$. 

\begin{Lem}\label{lem206}
For any bounded open set $\omega \subset \mathbb R^2$, there exists $\tilde{\zeta}^{\omega}=\tilde{\zeta}^{\omega}(d\beta) \in \mathcal M(I_+)$ such that
\begin{equation}\label{eq236}
\Big(\int_{\omega} e^{\tilde{w}_{k, \beta}}dx \Big)\mathcal P(d\beta) \overset{*}{\rightharpoonup} \tilde{\zeta}^{\omega}(d\beta) \quad {\rm in} \hspace{2mm} \mathcal M(I_+).
\end{equation}
Furthermore, there exists $\tilde{\psi}^{\omega} \in L^1(I_+, \mathcal P)$ such that 0 $\leq \tilde{\psi}^{\omega} \leq 1$ $\mathcal P$-a.e. on $I_+$ and
\begin{equation*}
\tilde{\zeta}^{\omega}(\eta)=\int_{\eta} \tilde{\psi}^{\omega}(\beta) \mathcal P(d\beta)
\end{equation*}
for any Borel set $\eta \subset I_+$.
\end{Lem}

\begin{proof}
Given bounded open set $\omega \subset \mathbb R^2$, we have
\begin{equation*}
\int_{I_+} \Big(\int_{\omega} e^{\tilde{w}_{k, \beta}} dx \Big) \mathcal P(d\beta) \leq 1.
\end{equation*}
Hence it holds that
\begin{equation} \label{eq238}
\Big(\int_{\omega} e^{\tilde{w}_{k, \beta}}dx \Big)\mathcal P(d\beta) \overset{*}{\rightharpoonup} \tilde{\zeta}^{\omega}(d\beta) \quad {\rm in} \hspace{2mm} \mathcal M(I_+).
\end{equation}
Now we shall show that the limit measure $\tilde{\zeta}^{\omega}=\tilde{\zeta}^{\omega}(d\beta) \in \mathcal M(I_+)$ is absolutely continuous with respect to $\mathcal P$.\

Let $\eta \subset I_+$ be a Borel set and $\epsilon>0$. Then each compact set $K \subset \eta$ admits an open set $J \subset I_+$ such that 
\begin{equation*}
K \subset \eta \subset J, \quad \mathcal P(J) \leq \epsilon+ \mathcal P(K).
\end{equation*}
Now we take $\varphi \in C(I_+)$ satisfying
\begin{equation*}
\varphi=1 \quad {\rm on} \hspace{2mm} K, \quad 0 \leq \varphi \leq 1 \quad {\rm on} \hspace{2mm} I_+, \quad {\rm supp}\varphi \subset J.
\end{equation*}
Then (\ref{eq238}) implies 
\begin{align*}
\tilde{\zeta}^{\omega}(K) &=\int_{K} \tilde{\zeta}^{\omega}(d\beta) \leq \int_{I_+} \varphi(\beta) \tilde{\zeta}^{\omega}(d\beta)\\
&=\lim_{k \to \infty} \int_{I_+} \varphi(\beta) \Big( \int_{\omega} e^{\tilde{w}_{k, \beta}}\Big)\mathcal P(d\beta) \leq \int_{I_+} \varphi(\beta) \mathcal P(d\beta)\\
& \leq \int_{J} \mathcal P(d\beta)=\mathcal P(J) \leq \epsilon+ \mathcal P(\eta),
\end{align*}
and therefore
\begin{equation*}
0 \leq \tilde{\zeta}^{\omega}(\eta) = \sup\{ \tilde{\zeta}^{\omega}(K) \mid K \subset \eta : {\rm compact} \} \leq \epsilon+\mathcal P(\eta).
\end{equation*}
This shows the absolute continuity of $\tilde{\zeta}^{\omega}$ with respect to $\mathcal P$. Therefore, by the Radon-Nikod$\rm \acute{y}$m theorem, there exists $\tilde{\psi}^{\omega} \in L^1(I_+, \mathcal P)$ such that 0 $\leq \tilde{\psi}^{\omega} \leq 1$ $\mathcal P$-a.e. on $I_+$ and
\begin{equation*}
\tilde{\zeta}^{\omega}(\eta)=\int_{\eta} \tilde{\psi}^{\omega}(\beta) \mathcal P(d\beta)
\end{equation*}
for any Borel set $\eta \subset I_+$. 
\end{proof}

\begin{Prop} \label{pro203}
There exists $\tilde{\psi} \in L^1(I_+, \mathcal P)$ and $0 \leq \tilde{\psi}(\beta) \leq 1$ $\mathcal P$-a.e $\beta$ such that
\begin{equation} \label{eq240}
\int_{\mathbb R^2} \tilde{f} dy=\overline\lambda \int_{I_+} \beta \tilde{\psi}(\beta) \mathcal P(d\beta).
\end{equation}
\end{Prop}
\begin{proof}
$\omega$ and $\tilde{\psi}^{\omega}$ as in Lemma \ref{lem206}. Taking $R_j \uparrow +\infty$ and $\omega_j=B_{R_j}$, by the monotonicity of $\tilde{\psi}^{\omega}$ with respect to $\omega$, there exists $\tilde{\zeta} \in \mathcal M(I_+)$ and $\tilde{\psi} \in L^1(I_+, \mathcal P)$ such that 
\begin{align*}
0&\leq \tilde{\psi}(\beta) \leq 1, \quad \mathcal P{\mathchar `-}a.e. \hspace{1mm}\beta\\
0&\leq \tilde{\psi}^{\omega_1}(\beta)\leq \tilde{\psi}^{\omega_2}(\beta)\leq\cdot\cdot\cdot \to \tilde{\psi}(\beta), \quad \mathcal P{\mathchar `-}a.e. \hspace{1mm} \beta\\
\tilde{\zeta}(\eta)&=\int_{\eta} \tilde{\psi}(\beta)\mathcal P(d\beta)\quad {\rm for \hspace{1mm}any \hspace{1mm}Borel \hspace{1mm}set}\hspace{1mm}\eta \subset I_+.
\end{align*}
First, (\ref{eq210}) implies
\begin{equation*}
\overline\lambda \int_{I_+} \beta\tilde{\psi}^{\omega_j}(\beta) \mathcal P(d\beta)=\lim_{k \to \infty} \lambda_k \int_{I_+} \beta \Big(\int_{\omega_j} e^{\tilde{w}_{k, \beta}} dx \Big) \mathcal P(d\beta)=\int_{\omega_j} \tilde{f}.
\end{equation*}
Then we obtain
\begin{equation} \label{eq240'}
\beta_0:= \int_{\mathbb R^2} \tilde{f} = \overline\lambda \int_{I_+} \beta \tilde{\psi}(\beta) \mathcal P(d\beta)
\end{equation}
{by the monotone convergence theorem.}
\end{proof}

Let
\begin{equation} \label{eq230}
\mathcal B=\{\beta \in {\rm supp} \mathcal P \mid \limsup_{k \to \infty} (w_{k, \beta}(0)-w_k(0))>-\infty \}.
\end{equation}
From the proof of Proposition \ref{pro201}, it follows that if $\mathcal P(\mathcal B)=0$ then $\tilde{f} \equiv 0$, a contradiction. Hence $\mathcal P(\mathcal B)>0$, and the value 
\begin{equation} \label{eq232}
\beta_{inf}=\inf_{\beta \in \mathcal B} \beta
\end{equation}
is well-defined. Then we find 
\begin{equation} \label{eq234}
\mathcal B=I_{inf} \cap {\rm supp}\mathcal P
\end{equation}
by the monotonicity (\ref{eq200}), where
\begin{equation*}
  I_{inf} = \begin{cases}
    [\beta_{inf}, 1] &{\rm if} \hspace{2mm} \beta_{inf}\in  \mathcal B,\\
    (\beta_{inf}, 1] &{\rm if} \hspace{2mm} \beta_{inf}\not\in  \mathcal B.
  \end{cases}
\end{equation*}
\begin{Lem} \label{lem208}
For any $\beta \in I_{inf}$, it holds that
\begin{equation*}
\beta> \frac{4\pi}{\beta_0}.
\end{equation*}
\end{Lem}
\begin{proof}
By the definition, every $\beta \in \mathcal B$ admits a subsequence such that $\tilde{w}_{k, \beta}(0)=w_{k, \beta}(0)-w_k(0)=O(1)$. From (\ref{eq216}), $\tilde{w}_{k, \beta}$ satisfies
\begin{equation*}
-\Delta\tilde{w}_{k, \beta}=\beta(-\Delta\tilde{w}_{k})=\beta \tilde{f}_k.
\end{equation*}
By the argument developed for the proof of (\ref{eq210})-(\ref{eq214}), we have $\tilde{w}_{\beta}=\tilde{w}_{\beta}(x) \in C^2(\mathbb R^2)$ such that
\begin{equation*}
\tilde{w}_{k, \beta} \to \tilde{w}_{\beta} \quad {\rm in} \hspace{2mm} C^2_{loc}(\mathbb R^2).
\end{equation*}
The limit $\tilde{w}_{\beta}$ satisfies
\begin{equation*}
-\Delta \tilde{w}_{\beta}=\beta \tilde{f}, \quad \tilde{w}_{\beta} \leq \tilde{w}_{\beta}(0)=0 \quad {\rm in }\hspace{2mm} \mathbb R^2, \quad \int_{\mathbb R^2} e^{\tilde{w}_{\beta}}\leq 1
\end{equation*}
and 
\begin{equation} \label{eq242}
\tilde{w}_{\beta}(x) \geq -\beta \frac{\beta_0}{2\pi}\log(1+|x|)+\frac{\beta}{2\pi}\int_{\mathbb R^2} \tilde{f}(y) \log\frac{|y|}{1+|y|}
\end{equation}
with $\tilde{f}=\tilde{f}(x)$ given in Proposition \ref{pro201}.\

Since $\tilde{f} \in L^1 \cap L^{\infty}(\mathbb R^2)$ and $\int_{\mathbb R^2} e^{\tilde{w}_\beta}<+\infty$ for any $\beta \in I_{inf}$, we obtain $\beta>4\pi/\beta_0$. 
\end{proof}

Similarly to \cite{key010}, on the other hand, we have the following lemma, where $(r, \theta)$ denotes the polar coordinate in $\mathbb R^2$.
\begin{Lem} \label{lem210}
We have 
\begin{equation*}
\lim_{r\to+\infty}r\tilde{w}_r=-\frac{\beta_0}{2\pi}, \quad \lim_{r\to+\infty}\tilde{w}_{\theta}=0
\end{equation*}
uniformly in $\theta$.
\end{Lem}
\begin{proof}
From (\ref{eq224}) and (\ref{eq228}), it follows that
\begin{align*}
r\tilde{w}_r(x)&=-\frac{\beta_0}{2\pi}-\frac{1}{2\pi}\int_{\mathbb R^2} \frac{y \cdot (x-y)}{|x-y|^2}\tilde{f}(y) dy,\\
\tilde{w}_{\theta}(x)&=\frac{1}{2\pi}\int_{\mathbb R^2} \frac{\overline y \cdot (x-y)}{|x-y|^2}\tilde{f}(y) dy, \quad \overline y=(y_2, -y_1).
\end{align*}
Hence it suffices to show 
\begin{equation*}
\lim_{|x| \to +\infty} I_1(x)=\lim_{|x| \to +\infty} I_2(x)=0,
\end{equation*}
where
\begin{equation*}
I_1(x)=\int_{|x-y|>|x|/2} \frac{|y|}{|x-y|}\tilde{f}(y)dy, \quad I_2(x)=\int_{|x-y|\leq |x|/2} \frac{|y|}{|x-y|}\tilde{f}(y)dy.
\end{equation*}
Since $\tilde{f} \in L^1(\mathbb R^2)$, we have $\displaystyle \lim_{|x|\to+\infty}I_1(x)=0$ by the dominated convergence theorem.\

Next, (\ref{eq210}) implies
\begin{align*}
I_2(x)&=\lim_{k \to \infty} \int_{|x-y|\leq|x|/2} \frac{|y|}{|x-y|} \Big(\lambda_k \int_{I_+} \beta e^{\tilde{w}_{k, \beta}(y)}\mathcal P(d\beta)\Big)dy\\
&={\overline\lambda} \lim_{k \to \infty}\int_{[\beta_{inf}, 1)} \beta \Big(\int_{|x-y|\leq|x|/2} \frac{|y|}{|x-y|}e^{\tilde{w}_{k, \beta}} dy\Big)\mathcal P(d\beta),
\end{align*}
recalling (\ref{eq230}) and (\ref{eq232}). Now we use (\ref{eq216}), (\ref{eq218}) and (\ref{eq210}) with (\ref{eq228}), to confirm
\begin{equation} \label{eq251}
\tilde{w}_{k, \beta}(x) \leq \beta \tilde{w}_k(x) =\beta(-\tilde{z}(x)+\tilde{z}(0))+o(1)
\end{equation}
as $k \to \infty$, locally uniformly in $x \in \mathbb R^2$. Hence it holds that
\begin{equation*}
0 \leq I_2(x) \leq C_4\int_{|x-y|\leq |x|/2} \frac{|y|}{|x-y|} \cdot \int_{[\beta_{inf}, 1]} e^{-\beta \tilde{z}(y)}\mathcal P(d\beta)dy.
\end{equation*}
Then (\ref{eq226}) and Lemma \ref{lem208} imply
\begin{equation*}
0 \leq I_2(x) \leq C_5 |x|^{-(1+\epsilon_0)} \int_{|x-y|\leq |x|/2} \frac{dy}{|x-y|} \leq C_6 |x|^{-(2+\epsilon_0 )}
\end{equation*}
with some $\epsilon_0>0$, where we have used 
\begin{equation*}
|x-y| \leq \frac{|x|}{2} \Rightarrow \frac{1}{2} \leq |y| \leq \frac{3}{2}
\end{equation*}
Hence $\lim_{|x|\to\infty} I_2(x)=0$ follows.

\end{proof}

The Pohozaev identity 
\begin{align}\label{eq253}
R\int_{\partial B_R}\frac{1}{2}|\nabla u|^2-u_r^2 ds&= R\int_{\partial B_R} A(x) F(u) ds \nonumber \\
&-\int_{B_R} 2A(x)F(u)+F(u)(x \cdot \nabla A(x)) dx
\end{align}
is valid to $u=u(x) \in C^2(\overline B_R)$ satisfying
\begin{equation}\label{eq255}
-\Delta u=A(x)F'(u) \quad {\rm in} \hspace{2mm} B_R,
\end{equation}
where $F \in C^1(\mathbb R)$, $A \in C^1(\overline B_R)$, and $ds$ denote the surface element on the boundary. By this identity and Lemma {\ref{lem210}}, we obtain the following fact.
\begin{Lem} \label{lem212}
It holds that
\begin{equation} \label{eq257}
\int_{I_+} \tilde{\psi}(\beta) \mathcal P(d\beta)=\frac{\overline\lambda}{8\pi}\Big(\int_{I_+} \beta \tilde{\psi}(\beta) \mathcal P(d\beta)\Big)^2.
\end{equation}
\end{Lem}
\begin{proof}
We apply (\ref{eq253}) for (\ref{eq255}) to (\ref{eq205}) where $u=\tilde{w}_k$ and
\begin{equation*}
F(\tilde{w}_k)=\lambda_k \int_{I_+} e^{\tilde{w}_{k, \beta}}\mathcal P(d\beta), \quad A(x) \equiv 1.
\end{equation*}
It follows that 
\begin{align} \label{eq259}
R\int_{\partial B_R}\frac{1}{2}|\nabla \tilde{w}_k|^2-(\tilde{w}_k)_r^2 ds&=-2\lambda_k \int_{I_+} \Big(\int_{B_R} e^{\tilde{w}_{k, \beta}} dx \Big)\mathcal P(d \beta) \nonumber \\
&+R\lambda_k \int_{I_+} \Big(\int_{\partial B_R} e^{\tilde{w}_{k, \beta}} \Big)\mathcal P(d \beta){.}
\end{align}
By Lemma \ref{lem210}, we have
\begin{equation*}
[L. H. S \hspace{1mm} of \hspace{1mm} (\ref{eq259})] \to -\pi \bigg(\frac{\beta_0}{2\pi}\bigg)^2 \quad as\hspace{2mm}  k \to \infty\hspace{2mm} and \hspace{2mm}R \to \infty.
\end{equation*}
The second term of right hand side of (\ref{eq259}) tends to 0 as  $k \to \infty$ and $R \to \infty$. Indeed, we have 
\begin{align} \label{eq261}
\int_{I_+} \Big(\int_{\partial B_R} e^{\tilde{w}_{k, \beta}} \Big)\mathcal P(d \beta)&=\int_{I_{inf}} \Big(\int_{\partial B_R} e^{\tilde{w}_{k, \beta}} \Big)\mathcal P(d \beta) \nonumber \\
&+\int_{I_+ \setminus I_{inf}}\Big(\int_{\partial B_R} e^{\tilde{w}_{k, \beta}} \Big)\mathcal P(d \beta).
\end{align}
Thanks to Lemma \ref{lem208} and (\ref{eq226}), the first term of the right hand side of (\ref{eq261}) tends to 0. And the second term also so because of the definition of $I_{inf}$.
Therefore, we have
\begin{equation*}
-\pi\Big(\frac{\beta_0}{2\pi}\Big)^2=-2\overline\lambda \int_{I_+} \tilde{\psi}(\beta) \mathcal P(d\beta),
\end{equation*}
and imply that (\ref{eq257}) holds. 

\end{proof}

\section{Proof of Theorem \ref{th102}}
By our assumption, residual vanishing occurs and we have
\begin{equation} \label{eq263}
\int_{\mathbb R^2} \tilde{f} dx=m(x_0)=\overline \lambda \int_{I_+} \alpha \mathcal P(d\alpha).
\end{equation}
Moreover from (\ref{eq263}), (\ref{eq240'}) and $0 \leq \tilde{\psi}(\beta) \leq 1$ $\mathcal P{\mathchar `-}$a.e. on $I_+$, it follows that
\begin{equation} \label{eq265}
\tilde{\psi}(\beta)=1 \quad \mathcal P{\mathchar `-} a.e. \hspace{1mm} on \hspace{1mm} I_+.
\end{equation}

\begin{Prop} \label{pro205}
Under the assumptions of Theorem \ref{th102}, it follows that
\begin{equation}\label{eq243}
\alpha_{min}=\beta_{inf}\in \mathcal B, \quad \alpha_{min}>\frac{4\pi}{\beta_0}.
\end{equation}
\end{Prop}
\begin{proof}
First of all, we have
\begin{equation}\label{eq245}
\alpha_{min}=\beta_{inf}.
\end{equation}
Indeed, $\beta_{inf} \geq \alpha_{\rm \min}$ is obvious, we assume the contrary, $\beta_{inf} > \alpha_{\rm \min}$. Then it holds that supp $\tilde{\psi} \subset$ $[\beta_{inf}, 1]$ by the definition of $\beta_{inf}$ and $\tilde{\psi}$, thus we obtain $\mathcal P([\alpha_{\rm min}, (\beta_{inf}+\alpha_{\rm min})/2])>0$ and $\tilde{\psi}=0$ $\mathcal P$-a.e. on $[\alpha_{\rm min}, (\beta_{inf}+\alpha_{\rm min})/2]$. However, this is impossible by (\ref{eq265}).\

Next, it holds that
\begin{equation} \label{eq247}
\alpha_{min} \in \mathcal B.
\end{equation}
If not, $\alpha_{min} \not\in \mathcal B$, by our assumption of $\mathcal P(\{\alpha_{min}\})>0$,
\begin{equation*}
\tilde{f}_k(x)=\lambda_k \Big( \int_{I_+ \setminus \{\alpha_{min}\}} \alpha e^{\tilde{w}_{k, \alpha}(x)} \mathcal P(d\alpha) +\alpha_{min} e^{\tilde{w}_{k, \alpha_{min}}(x)} \mathcal P(\{\alpha_{min}\}) \Big),
\end{equation*}
and then passing to a subsequence, by Proposition \ref{pro201} and the definition of $\mathcal B$,
\begin{equation} \label{eq249}
\tilde{f}(x)= \overline\lambda \int_{I_+ \setminus \{\alpha_{min}\}} \alpha e^{\tilde{w}_{\alpha}(x)} \mathcal P(d\alpha)
\end{equation}
From (\ref{eq249}) and the integral condition $\int_{\mathbb R^2} e^{\tilde{w_{\alpha}}} \leq 1$, we have
\begin{align*}
\int_{\mathbb R^2} \tilde{f}(x) dx&=\overline\lambda \int_{I_+ \setminus \{\alpha_{min}\}} \alpha \int_{\mathbb R^2} e^{\tilde{w}_{\alpha}(x)} dx \mathcal P(d\alpha)\\
& \leq \overline\lambda \int_{I_+ \setminus \{\alpha_{min}\}} \alpha \mathcal P(d\alpha) < \overline\lambda \int_{I_+} \alpha \mathcal P(d\alpha),
\end{align*}
which is a contradiction to (\ref{eq263}). Therefore (\ref{eq247}) holds.\

Finally, (\ref{eq247}) and Lemma \ref{lem208} imply that $\alpha_{min}>4\pi /\beta_0$ holds. Now we complete the proof of Proposition \ref{pro205}.
\end{proof}
Let $G=G(x, y)$ be the Green function:
\begin{equation*}
-\Delta_x G(x, y)=\delta_y \quad in \hspace{2mm} \Omega, \quad G(\cdot, y)=0 \quad on \hspace{2mm} \partial \Omega.
\end{equation*}

\begin{Prop} \label{prop300}
Under the assumption of Theorem \ref{th102}, it follows that
\begin{equation*}
\lim_{k \to \infty} \int_{\Omega} e^{\alpha_{min} v_k} dx=+\infty.
\end{equation*}
\end{Prop}
\begin{proof}
By an argument of \cite{key00}, we shall establish the desired result. \

Indeed, setting $\omega=B_{R_0}(x_0)$ and note that $v_k \geq C_0$ on $\partial \omega$ and $z_k$ be a solution of
\begin{equation*}
-\Delta z_k=\mu_k \quad in \hspace{2mm} \omega, \quad z_k=C_0 \quad on \hspace{2mm} \partial \omega
\end{equation*}
where $\mu_k$ in (\ref{eq114}). By the maximum principle, we have 
\begin{equation} \label{eq300'}
v_k \geq z_k \quad in \hspace{2mm} \omega.
\end{equation}
On the other hand $z_k \to z$ locally uniformly in $\omega \setminus \{x_0 \}$ and 
\begin{equation*}
\mu_k(dy) \overset{*}{\rightharpoonup} s'(y)dy+m (x_0)\delta_{x_0}(dy) \quad in \hspace{2mm} \mathcal M(\overline \omega),
\end{equation*}
where $s' \in L^1(B_{R_0})$ is nonnegative. Therefore,
\begin{equation*}
-\Delta z=\mu \quad in \hspace{2mm} \omega, \quad z=C_0 \quad on \hspace{2mm} \partial\omega,
\end{equation*}
and
\begin{equation} \label{eq300}
z(x) \geq \frac{m(x_0)}{2\pi} \log\frac{1}{|x-x_0|}-C, \quad x \in \overline \omega \setminus \{x_0\}.
\end{equation}
By Proposition \ref{pro205}, we have
\begin{equation} \label{eq301}
\alpha_{min}=\beta_{inf} > \frac{4\pi}{\beta_0}.
\end{equation}
By (\ref{eq300'})-(\ref{eq301}) and (\ref{eq120}) we obtain 
\begin{equation} \label{eq302}
\alpha_{min} \big(\liminf_{k\to\infty} v_k(x) \big)\geq \log\frac{1}{|x-x_0|^2}-C', \quad x \in \overline \omega \setminus \{x_0\}.
\end{equation}
Therefore from Fatou's lemma and the definition of $\alpha_{min}$, 
\begin{equation} \label{eq303}
\liminf_{k \to \infty}\int_{\Omega} e^{\alpha_{min} v_k} \geq C'' \int_{\omega} \frac{1}{|x-x_0|^2} dx=+\infty.
\end{equation}
\end{proof}

\begin{Lem} \label{lem300}
It holds that
\begin{equation} \label{eq304}
w_k(x)+\log\int_{\Omega}e^{v_k} \to \overline\lambda \int_{I_+} \alpha \mathcal P(d\alpha) G(\cdot+x_0, x_0) \quad in \hspace{2mm} C_{loc}^2(B_{3R_0} \setminus \{0\}),
\end{equation}
as $k \to \infty$. For every $\omega \subset \subset B_{3R_0} \setminus \{0\}$, there exists $C_{1, \omega}>0$ such that
\begin{equation} \label{eq306}
\osc_{\omega} w_{k} \equiv \sup_{\omega} w_k -\inf_{\omega} w_k \leq C_{1, \omega}
\end{equation}
for large $k \in \mathbb N$.
\end{Lem}

\begin{proof}
Without loss of generality, we may assume that $B_{3R_0} \subset \Omega_k \equiv \Omega-\{x_k\}$ for large $k \in \mathbb N$. By the definition of $w_k$, we have
\begin{align*}
w_k(x)+\log\int_{\Omega}e^{v_k}&=\int_{\Omega} G(x+x_k, y)\mu_k(dy),
\end{align*}
for $x \in B_{3R_0}$ where $\mu_k(dy)$ as in Proposition \ref{prop100}. Note that
\begin{align} \label{eq306'}
w_k(x)+\log\int_{\Omega}e^{v_k}&=G(x+x_k, x_0) \int_{\Omega} \mu_k(dy) \nonumber \\
&+\int_{\Omega} [G(x+x_k, y) -G(x+x_k, x_0)]\mu_k(dy). 
\end{align}
From the fact that $[G(x+x_k, y) -G(x+x_k, x_0)] \to 0$ as $y \to x_0$ locally uniformly for $x \in B_{3R_0} \setminus \{0\}$ and (\ref{eq303}), the second term of right hand side in the above relation tends to 0 as $k \to \infty$. Indeed, we set for any $r \in (0, R_0)$,
\begin{align*}
&\int_{\Omega} [G(x+x_k, y) -G(x+x_k, x_0)]\mu_k(dy)\\
&=\int_{B_r(x_0)} [G(x+x_k, y) -G(x+x_k, x_0)]\mu_k(dy)\\
&+\int_{\Omega \setminus B_r(x_0)} [G(x+x_k, y) -G(x+x_k, x_0)]\mu_k(dy)\\
&=: I_1^{k, r}(x)+I_2^{k, r}(x).
\end{align*}
By the direct calculation and (\ref{eq303}), we have for $r>0$,
\begin{align} \label{eq307}
|I_2^{k, r}(x)|& \leq \int_{\Omega \setminus B_r(x_0)} |G(x+x_k, y) -G(x+x_k, x_0)| \lambda_k \int_{I_+} \frac{\alpha e^{\alpha v_k}}{\int_{\Omega} e^{\alpha v_k} dx} \mathcal P(d\alpha) dy \nonumber \\
&\leq \frac{\lambda_k}{\int_{\Omega} e^{\alpha_{min} v_k} dx} \int_{\Omega \setminus B_r(x_0)} |G(x+x_k, y) -G(x+x_k, x_0)| e^{v_k}dy \nonumber \\
& \to 0
\end{align}
as $k \to \infty$ locally uniformly in $x \in B_{3R_0} \setminus \{0\}$. On the other hands, for large $k \in \mathbb N$,
\begin{align} \label{eq307'}
|I_1^{k, r}(x)|& \leq \int_{B_r(x_0)} |G(x+x_k, y) -G(x+x_k, x_0)| \mu_k(dy) \nonumber \\
&\leq \sup_{B_r(x_0)}|G(x+x_k, y) -G(x+x_k, x_0)| \nonumber \\
& \to 0
\end{align}
as $r\to0$ locally uniformly in $x \in B_{3R_0} \setminus \{0\}$. From (\ref{eq306'})-(\ref{eq307'}), it follows that
\begin{equation*} 
w_k(x)+\log\int_{\Omega}e^{v_k} \to \lambda_0 \int_{I_+} \alpha \mathcal P(d\alpha) G(\cdot+x_0, x_0),
\end{equation*}
as $k \to \infty$ locally uniformly in $x \in B_{3R_0} \setminus \{0\}$. Furthermore, we have
\begin{equation*}
\frac{\partial}{\partial x_{ij}} \big(w_k(x)+\log\int_{\Omega}e^{v_k} \big)=\int_{\Omega} \frac{\partial}{\partial x_{ij}} G(x+x_k, y)\mu_k(dy),
\end{equation*}
for $i, j=1, 2$ so that by the same argument here we obtain (\ref{eq304}). (\ref{eq306}) is the direct consequence of (\ref{eq304}). 
\end{proof}

Note that $B_{2R_0} \subset \Omega-\{x_k\}$ for large $k \in \mathbb N$. We decompose $w_k$ as $w_k=w_k^{(1)}+w_k^{(2)}$, using the solutions $w_k^{(1)}$ and $w_k^{(2)}$ to
\begin{align*}
-\Delta w_k^{(1)}&=g_k \quad {\rm in} \hspace{2mm}B_{2R_0}, \quad w_k^{(1)}=0 \quad \hspace{2.8mm} {\rm on} \hspace{2mm} \partial B_{2R_0},\\
-\Delta w_k^{(2)}&=0 \quad \hspace{1.8mm} {\rm in} \hspace{2mm}B_{2R_0}, \quad w_k^{(2)}=w_k \quad {\rm on} \hspace{2mm} \partial B_{2R_0}.
\end{align*}
where
\begin{equation*}
g_k=g_k(x)=\lambda_k \int_{I_+} \alpha e^{w_{k, \alpha}(y)} \mathcal P(d\alpha),
\end{equation*}
and $R_0>0$ as in (\ref{eq209}). By the maximum principle and Lemma \ref{lem300}, we also have $C_2>0$ independent of $k$ such that
\begin{equation*}
\osc_{ \overline B_{2R_0}} w_k^{(2)}\leq C_2.
\end{equation*}
Thus it holds that
\begin{equation} \label{eq308}
w_k(x)-w_k(0)=w_k^{(1)}(x)-w_k^{(1)}(0)+O(1)
\end{equation}
as $k \to \infty$ uniformly in $x \in B_{2R_0}$.\

Let $G_0=G_0(x, y)$ be the another Green function defined by
\begin{equation*}
-\Delta_x G_0(\cdot, y)=\delta_y \quad {\rm in} \hspace{2mm} B_{2R_0}, \quad G_0(\cdot, y)=0 \quad {\rm on} \hspace{2mm} \partial B_{2R_0}.
\end{equation*}
Then it holds that 
\begin{equation} \label{eq310}
w_k^{(1)}(x)-w_k^{(1)}(0)=\int_{B_{2R_0}} (G_0(x, y)-G_0(0, y))g_k(y) dy
\end{equation}
for $x \in B_{2R_0}$. We have, more precisely,
\begin{equation*}
  G_0(x, y) = \begin{cases}
    \Gamma(|x-y|)-\Gamma(\frac{|y|}{2R_0}|x-\overline y|), &{\rm if} \hspace{2mm} y\neq 0, \hspace{1mm}y\neq x,\\
    \Gamma(|x|)-\Gamma(2R_0) &{\rm if} \hspace{2mm} y=0, \hspace{1mm}y\neq x,
  \end{cases}
\end{equation*}
using the fundamental solution and the Kelvin transformation:
\begin{equation*}
\Gamma(|x|)=\frac{1}{2\pi}\log \frac{1}{|x|}, \quad \overline y=\Big(\frac{2R_0}{|y|} \Big)^2 y,
\end{equation*}
which implies
\begin{equation*}
G_0(x, y)-G_0(0, y)=\frac{1}{2\pi}\log \frac{|y|}{|x-y|}-\frac{1}{2\pi}\log \frac{|\overline y|}{|x-\overline y|}
\end{equation*}
for $y \in B_{2R_0}$ satisfying $y \neq x$ and $y \neq 0$.\

By 
\begin{equation*}
\frac{2}{3} \leq \frac{|\overline y|}{|x-\overline y|} \leq 2, \quad x \in B_{R_0}, \quad y\in B_{2R_0}\setminus\{0\},
\end{equation*}
and 
\begin{equation*}
0 \leq \int_{B_{2R_0}} g_k \leq \lambda_k \int_{I_+} \alpha \mathcal P(d\alpha)=O(1),
\end{equation*}
we end up with
\begin{equation} \label{eq312}
\int_{B_{2R_0}} (G_0(x, y)-G_0(0, y))g_k(y) dy =\frac{1}{2\pi} \int_{B_{2R_0}} g_k(y) \log \frac{|y|}{|x-y|} dy+O(1)
\end{equation}
as $k \to \infty$ uniformly in $x \in B_{R_0}$.\

Consequently, (\ref{eq308})-(\ref{eq312}) yield
\begin{equation*} 
w_k(x)-w_k(0)=\frac{1}{2\pi} \int_{B_{2R_0}} g_k(y) \log \frac{|y|}{|x-y|} dy+O(1)
\end{equation*}
as $k \to \infty$ uniformly in $x \in B_{R_0}$. This means \\
\begin{align} \label{eq313}
\tilde{w}_k(x)&=\frac{1}{2\pi} \int_{B_{2R_0}}  g_k(y) \log \frac{|y|}{|\sigma_k x-y|} dy+O(1) \nonumber \\ 
&=\frac{1}{2\pi} \int_{B_{2R_0 \sigma_k^{-1}}} \tilde{f}_k(z) \log \frac{|z|}{|x-z|} dz+O(1)
\end{align}
as $k \to \infty$ uniformly in $x \in B_{R_0 \sigma_k^{-1}}$, where $\tilde{f}_k=\lambda_k \int_{I_+} \alpha e^{\tilde{w}_{k, \alpha}} \mathcal P(d\alpha)$.\

Let $\beta_0$ be as in (\ref{eq240'}), and put
\begin{equation} \label{eq314}
\beta_k:=\int_{B_{2R_0 \sigma_k^{-1}}} \tilde{f}_k.
\end{equation}
To employ the argument of \cite{key030}, we prepare the following lemma.

\begin{Lem} \label{lem302}
For any $\epsilon>0$, there exists $k_\epsilon \in \mathbb N$ and $L_\epsilon >0$such that 
\begin{equation} \label{eq316}
\int_{B_{2R_0 \sigma_k^{-1}} \setminus B_{L_{\epsilon}}} \tilde{f}_k dx \leq \epsilon
\end{equation}
for $k \geq k_\epsilon$.
\end{Lem}
\begin{proof}
For any $r>0$, setting
\begin{equation*}
\beta_k =\int_{B_{r}} \tilde{f}_k dx+ \int_{B_{2R_0 \sigma_k^{-1}} \setminus B_r} \tilde{f}_k dx =: I_1^{k, r} +I_2^{k, r}.
\end{equation*}
By (\ref{eq210}), we have
\begin{equation} \label{eq318}
\lim_{r \to \infty} \lim_{k \to \infty} I_1^{k, r}=\int_{\mathbb R^2} \tilde{f} dx.
\end{equation}
Moreover for any $k$, it holds that
\begin{equation} \label{eq320}
\beta_k \leq \lambda_k \int_{I_+} \alpha \mathcal P(d\alpha).
\end{equation}
Therefore from (\ref{eq318}), (\ref{eq320}) and (\ref{eq120}), we obtain
\begin{equation*}
I_2^{k, r}=o(1) \quad {\rm as} \hspace{2mm} k \to \infty, \quad r \to \infty
\end{equation*}
and get the desired result. 
\end{proof}

By the result of Lemma \ref{lem302}, we have
\begin{equation} \label{eq322}
\lim_{k \to \infty} \beta_k =\beta_0.
\end{equation}

\begin{Lem} \label{lem304}
For every $0<\epsilon \ll 1$, there exists ${R}_{\epsilon}>0$ and $C_{4, \epsilon}>0$ such that
\begin{equation} \label{eq324}
\tilde{w}_k(x) \leq -(\frac{\beta_k}{2\pi}-\epsilon)\log|x|+C_{4, \epsilon}
\end{equation}
for $k \gg 1$ and $x \in B_{R_0 {\sigma_k^{-1}}} \setminus B_{R_\epsilon}$.\
\end{Lem}
\begin{proof}
By (\ref{eq316}), given $0<\epsilon \ll 1$, there exists ${R}_\epsilon>0$, $k_{\epsilon} \in \mathbb N$ such that
\begin{equation} \label{eq326}
\frac{1}{2\pi} \int_{B_{R_\epsilon /2}} \tilde{f}_k dx \geq \frac{\beta_k}{2\pi}-\frac{\epsilon}{4}
\end{equation}
for $k \geq k_\epsilon$. It follows from (\ref{eq313}) that
\begin{equation} \label{eq328}
\tilde{w}_k(x)=K_k^1(x)+K_k^2(x)+K_k^3(x)+O(1), \quad k \to \infty,
\end{equation}
uniformly in $x \in B_{R_0 \sigma_k^{-1}} \setminus B_{R_\epsilon}$, where
\begin{align*}
K_k^1(x) &=\frac{1}{2\pi}\int_{B_{R_\epsilon /2}} \tilde{f}_k(y) \log\frac{|y|}{|x-y|}dy,\\
K_k^2(x) &=\frac{1}{2\pi}\int_{B_{|x|/2}(0)} \tilde{f}_k(y) \log\frac{|y|}{|x-y|}dy, \\
K_k^3(x) &=\frac{1}{2\pi}\int_{B'(x)} \tilde{f}_k(y) \log\frac{|y|}{|x-y|}dy,
\end{align*}
for $B'(x)=B_{R_0 \sigma_k^{-1}} \setminus (B_{R_\epsilon/2} \cap B_{|x|/2}(x))$.\

Since
\begin{equation*}
\frac{|y|}{|x-y|} \leq 2\frac{|y|}{|x|} \leq \frac{R_\epsilon}{|x|}, \quad y \in B_{R_\epsilon/2}, \hspace{2mm} x \in B_{R_0 \sigma_k^{-1}}\setminus B_{R_\epsilon},
\end{equation*}
there exists $C_{5, \epsilon}>0$ such that
\begin{align} \label{eq330}
K_k^1(x)  &\leq \frac{1}{2\pi} (\log R_\epsilon -\log |x|) \int_{B_{R_\epsilon /2}} \tilde{f}_k(y) \nonumber \\
&\leq C_{5, \epsilon}-(\frac{\beta_k}{2\pi }- \frac{\epsilon}{4}) \log|x|
\end{align}
for $k \geq k_\epsilon$ and $x \in B_{R_0 \sigma_k^{-1}}\setminus B_{R_\epsilon}$ by (\ref{eq326}). We also have 
\begin{equation*}
\frac{|y|}{|x-y|}\leq 3, \quad y \in B_{R_0 \sigma_k^{-1}} \setminus {B_{|x|/2}(x)},
\end{equation*}
and hence
\begin{align} \label{eq332}
K_k^3(x) &\leq \frac{\log3}{2\pi} \int_{B'(x)} \tilde{f}_k \leq \frac{\log3}{2\pi}\| \tilde{f}_k \|_{L^1(B_{2R_0 \sigma_k^{-1}})} \nonumber \\
& \leq \frac{\lambda_k \log3}{2\pi} \int_{I_+} \alpha \mathcal P(d\alpha),
\end{align}
for large $k$ and $x \in B_{R_0 \sigma_k^{-1}} \setminus B_{R_\epsilon}$.\

Now we take
\begin{equation*}
D_1(x)=B_{|x|^{-1}}(x), \quad D_2(x)=B_{|x|/2}(x) \setminus B_{|x|^{-1}}(x)
\end{equation*}
for $|x| > R_{\epsilon} \geq \sqrt{2}$. Since
\begin{equation*}
|y|<|x|+1/|x|, \quad y \in D_1(x)
\end{equation*}
and
\begin{equation*}
\frac{|y|}{|x-y|} \leq \frac{3}{2}|x|^2, \quad y \in D_2(x), \hspace{2mm} x \in B_{R_0 \sigma_k^{-1}} \setminus B_{R_\epsilon},
\end{equation*}
we have
\begin{align} \label{eq334}
&\frac{1}{2\pi}\int_{D_2(x)} \tilde{f}_k(y) \log\frac{|y|}{|x-y|}dy \nonumber \\
&\leq \frac{1}{2\pi}\int_{\frac{1}{2}|x|\leq |y| \leq \frac{3}{2} |x|} (2\log|x|+\log\frac{3}{2})\tilde{f}_k(y) dy\nonumber \\
&\leq \frac{1}{2\pi}\int_{\frac{R_\epsilon}{2}\leq |y| \leq \frac{3}{2} R_0 \sigma_k^{-1} } (2\log|x|+\log\frac{3}{2})\tilde{f}_k(y) dy \nonumber \\
&\leq \frac{1}{2} \epsilon \log|x|+O(1)
\end{align}
on the other hand,
\begin{align} \label{eq336}
&\frac{1}{2\pi}\int_{D_1(x)} \tilde{f}_k(y) \log\frac{|y|}{|x-y|}dy \nonumber \\
&=\frac{\|\tilde{f}_k \|_{L^\infty (D_1)}}{2\pi}\int_{D_1(x)}  \Big| \log\frac{1}{|x-y|} \Big| dy + \frac{1}{2\pi}\int_{D_1(x)} \tilde{f}_k(y)(\log |x| +C) dy \nonumber \\
& \leq C \int_0^{|x|^{-1}} r |\log r| dr+ \frac{1}{4}\epsilon \log |x| +O(1) \nonumber \\
&=\frac{1}{4}\epsilon \log |x| +O(1)
\end{align}
for $k \gg 1$, $x \in B_{R_0 \sigma_k^{-1}} \setminus B_{R_\epsilon}$. From (\ref{eq334}) and (\ref{eq336}), we get
\begin{equation} \label{eq338}
K_k^2(x) \leq \frac{3}{4}\epsilon \log|x|+C
\end{equation}
for $k \gg 1$, $x \in B_{R_0 \sigma_k^{-1}} \setminus B_{R_\epsilon}$.
From (\ref{eq330}), (\ref{eq332}), and (\ref{eq338}), we get the desired result. 
\end{proof}

\begin{Lem} \label{lem306}
It holds that
\begin{equation} \label{eq340}
\int_{B_{R_0 {\sigma_k}^{-1}}} \tilde{f}_k(y)|\log|y| |dy=O(1) \quad {\rm as} \hspace{2mm} k\to\infty.
\end{equation}
\end{Lem}
\begin{proof}
By $\lim_{k \to \infty} \beta_k=\beta_0$ and (\ref{eq243}), there exists $\epsilon_0>0$ and $\delta_0>0$ such that
\begin{equation} \label{eq342}
-\alpha_{min} (\frac{\beta_k}{2\pi}-\epsilon_0/2) \leq -(2+3\delta_0)
\end{equation}
for $k \gg 1$. Let
\begin{equation*}
R'_0=R_{\epsilon_0 /2}
\end{equation*}
for $R_\epsilon$ as in Lemma \ref{lem304} with $\epsilon=\epsilon_0/2$. Then, by (\ref{eq210})-(\ref{eq212}), (\ref{eq324}) and (\ref{eq342}) we obtain $C_{7, \epsilon_0}$ such that
\begin{align} \label{eq344}
\tilde{f}_k(y)&=\lambda_k \int_{I_{+}} e^{\tilde{w}_{k, \alpha}} \mathcal P(d\alpha) \nonumber \\
&\leq \lambda_k \int_{I_+} e^{\alpha\tilde{w}_{k}} \mathcal P(d\alpha) \nonumber \\
&\leq \lambda_k  {\rm exp}\Big[-\alpha_{min} \{(\beta_k/2\pi -\epsilon_0/2)\log|y|+C_{4, \epsilon_0} \}\Big] \nonumber \\
&\leq C_{7, \epsilon_0} |y|^{-(2+3\delta_0)}
\end{align}
for $k \gg 1$ and  $y \in B_{R_0 \sigma_k^{-1}} \setminus B_{R'_0}$.\

Therefore, we obtain $C_{8, \epsilon_0, \delta_0}>0$ such that
\begin{align*}
\int_{B_{R_0 {\sigma_k}^{-1}}} \tilde{f}_k(y)\big|\log|y| \big|dy &\leq \|\tilde{f} \|_{L^{\infty}(B_{R_0'})} \int_{B_{R'_0}} \big|\log|y| \big| dy \nonumber \\
&+ C_{7, \epsilon_0} \int_{B_{R_0 {\sigma_k}^{-1}} \setminus {B_{R'_0}}} |y|^{-(2+3\delta_0)}\log|y| dy \nonumber \\
&\leq C_{8, \epsilon_0, \delta_0}
\end{align*}
for $k \gg 1$, which means (\ref{eq340}). 
\end{proof}

\begin{Lem} \label{lem308}
There exists $\delta_0>0$ such that
\begin{equation*}
\tilde{w}_k(x)=-\frac{\beta_k}{2\pi} \log|x|+O(1) \quad {\rm as} \hspace{2mm} k\to\infty
\end{equation*}
uniformly in $x \in B_{R_0 \sigma_k^{-1}} \setminus B_{(\log\sigma_k^{-1})^{1/\delta_0}}$.
\end{Lem}
\begin{proof}
Let $\epsilon_0>0$ and $\delta_0>0$ as in (\ref{eq342}) and consider
\begin{equation} \label{eq346}
\beta_k^{\prime}(x)=\int_{B_{|x|/2}} \tilde{f}_k
\end{equation}
for $x \in B_{R_0 \sigma_k^{-1}} \setminus B_{(\log\sigma_k^{-1})^{1/\delta_0}}$ and $k \gg 1$.\\
First of all, 
\begin{equation} \label{eq348}
\big|\tilde{w}_k(x)+\frac{\beta_k}{2\pi} \log|x| \big|\leq (\frac{\beta_k}{2\pi}-\frac{\beta_k^{\prime}}{2\pi} )\log |x|+\big|\tilde{w}_k(x)+\frac{\beta_k^{\prime}(x)}{2\pi}\log|x| \big|
\end{equation}
for $x \in B_{R_0 \sigma_k^{-1}} \setminus B_{(\log\sigma_k^{-1})^{1/\delta_0}}$, $k \gg 1$. To get a estimate of right hand side of (\ref{eq348}), we divide this proof two steps.\

Step1. Since (\ref{eq344}) and (\ref{eq316}) hold, there exists $C_{9, \epsilon_0, \delta_0}>0$ such that
\begin{align} \label{eq350}
0 &\leq \beta_k-\beta_k^{\prime}(x) \leq \int_{B_{2R_0 \sigma_k^{-1}} \setminus B_{(\log\sigma_k^{-1})^{1/\delta_0}}} \tilde{f}_k \nonumber \\
&\leq C_{7, \epsilon_0} \int_{B_{R_0 \sigma_k^{-1}} \setminus B_{(\log\sigma_k^{-1})^{1/\delta_0}}} |y|^{-(2+3\delta_0)} dy + \int_{B_{2R_0 \sigma_k^{-1}} \setminus B_{R_0 \sigma_k^{-1}}} \tilde{f}_k\nonumber \\
&\leq C_{9, \epsilon_0, \delta_0} \sigma_k^{\frac{2}{\delta_0}+1}+o(1)
\end{align}
for $x \in B_{R_0 \sigma_k^{-1}} \setminus B_{(\log\sigma_k^{-1})^{1/\delta_0}}$, $k \gg 1$.\

Step2. By (\ref{eq313}),
\begin{align} \label{eq352}
\Big|\tilde{w}_k(x)+\frac{\beta_k^{\prime}(x)}{2\pi}\log|x| \Big | &\leq \Big| \frac{1}{2\pi}\int_{\frac{|x|}{2}<|y|<2R_0 \sigma_k^{-1}} \tilde{f}_k(y) \log\frac{|y|}{|x-y|} \Big|  \\
&+ \Big| \frac{1}{2\pi} \int_{B_{\frac{|x|}{2}}}  \tilde{f}_k(y) \log\frac{|x||y|}{|x-y|} +O(1)\Big| \label{eq354}
\end{align}
as $k \to \infty$.\

Note that if $z:= x/|y|$, $z_0:=y/|y|$ and $|z|<1/2$ then $1/2<|z-z_0|<3/2$, we have
\begin{align} \label{eq354'}
&\Big| \int_{\frac{|x|}{2}<|y|<2R_0 \sigma_k^{-1}} \tilde{f}_k(y) \log\frac{|y|}{|x-y|} \Big| \nonumber \\
&\leq \int_{\frac{|x|}{2}<|y|<2R_0 \sigma_k^{-1}} \tilde{f}_k(y) \Big| \log\frac{1}{|\frac{x}{|y|}-\frac{y}{|y|}|} \Big|\nonumber \\
&=\Bigg(\int_{\frac{1}{2}<\frac{|x|}{|y|}<2, |y|<2R_0 \sigma_k^{-1}}+\int_{0<\frac{|x|}{|y|}<\frac{1}{2}, |y|<2R_0 \sigma_k^{-1}}  \tilde{f}_k(y) \Big| \log\frac{1}{|\frac{x}{|y|}-\frac{y}{|y|}|} \Big|\Bigg) \nonumber \\
&\leq \int_{\frac{1}{2}<\frac{|x|}{|y|}<2, |y|<2R_0 \sigma_k^{-1}} \tilde{f}_k(y) \Big| \log\frac{1}{|\frac{x}{|y|}-\frac{y}{|y|}|}\Big|+\log2 \int_{2|x|<|y|<2R_0 \sigma_k^{-1}}  \tilde{f}_k(y).
\end{align}
Moreover, by (\ref{eq344}) and (\ref{eq316}), 
\begin{align} \label{eq355}
\int_{2|x|<|y|<2R_0 \sigma_k^{-1}} \tilde{f}_k(y)&\leq \int_{2(\log\sigma_k^{-1})^{\frac{1}{\delta}}<|y|<2R_0 \sigma_k^{-1}} \tilde{f}_k(y) \nonumber \\
&\leq \int_{2(\log\sigma_k^{-1})^{\frac{1}{\delta}}<|y|<R_0 \sigma_k^{-1}} \tilde{f}_k(y) + \int_{R_0 \sigma_k^{-1}<|y|<2R_0 \sigma_k^{-1}} \tilde{f}_k(y) \nonumber \\
&=o(1).
\end{align}
On the other hand, if $z:=y/|x|$ and $z_0:=x/|x|$ then we have
\begin{align} \label{eq355'}
&\int_{\frac{1}{2}<\frac{|x|}{|y|}<2, |y|<2R_0 \sigma_k^{-1}} \tilde{f}_k(y) \Big| \log\frac{1}{|\frac{x}{|y|}-\frac{y}{|y|}|}\Big|dy \nonumber \\
&=\int_{\frac{1}{2}<|z|<2, |y|<2R_0 \sigma_k^{-1}} \tilde{f}_k(|x|z) \Big| \log\frac{|z|}{|z-z_0|}\Big||x|^2 dz \nonumber \\
&\leq \int_{\frac{1}{2}<|z|<2} C_{7, \epsilon_0}(|x||z|)^{-(2+3\delta_0)} \Big| \log\frac{|z|}{|z-z_0|}\Big||x|^2 dz  \nonumber \\
&+ \int_{\frac{1}{2}<|z|<2, R_0 \sigma_k^{-1}<|y|<2R_0 \sigma_k^{-1}} \tilde{f}_k(|x|z) \Big| \log\frac{|z|}{|z-z_0|}\Big||x|^2 dz \nonumber \\
&\leq C'_{7, \epsilon_0} (\log\sigma_k^{-1})^{-3} + 5 \int_{R_0 \sigma_k^{-1}<|y|<2R_0 \sigma_k^{-1}} \tilde{f}_k(y)  dy \nonumber \\
&= o(1)
\end{align}
for $x \in B_{R_0 \sigma_k^{-1}} \setminus B_{(\log\sigma_k^{-1})^{1/\delta_0}}$, $k \gg 1$ by (\ref{eq344}) and (\ref{eq316}). \\
Therefore, from (\ref{eq354'})-(\ref{eq355'}), we obtain (\ref{eq352})$=o(1)$ as $k \to \infty$.\

Lastly, since $|y|<|x|/2$, $1/(2|y|) \leq|x-y| / (|x||y|) \leq 3/(2|y|)$ and (\ref{eq340}), it follows that
\begin{align*}
&\frac{1}{2\pi} \int_{|y|\leq \frac{|x|}{2}} \tilde{f}_k(y) \Big| \log\frac{|x||y|}{|x-y|}\Big| dy \\
&\leq \frac{1}{2\pi} \int_{|y|\leq R_0 \sigma_k^{-1}} \tilde{f}_k(y) (\log2+\log|y|) dy \\
&=O(1)
\end{align*}
as $k \to \infty$. Hence, $(\ref{eq354})=O(1)$ as $k \to \infty$.\

From Step1, Step2 and (\ref{eq348}), we complete the proof of Lemma \ref{lem308}.
\end{proof}

\noindent
{\bf Proof of Theorem \ref{th102}.} We take $\delta_0$ and $R'_0$ as in Lemma \ref{lem306}. First, (\ref{eq210}), (\ref{eq312}), and (\ref{eq314}) imply 
\begin{align} \label{eq356}
|\tilde{w}_k(x)+\frac{\beta_k}{2\pi} \log(1+|x|)| &\leq |\tilde{w}_k(x) |+\frac{\beta_k}{2\pi} \log(1+|x|) \nonumber \\
&\leq C_{12}
\end{align}
for $x \in B_{{R'_0}}$, while Lemma {\ref{lem308}} means 
\begin{align} \label{eq358}
\big|\tilde{w}_k(x)+\frac{\beta_k}{2\pi} \log(1+|x|) \big| \leq C_{13}, \quad x \in B_{R_0 \sigma_k^{-1}} \setminus B_{(\log \sigma_k^{-1})^{1/\delta_0}},
\end{align}
where $k \gg 1$.\

Now we put
\begin{align*}
\tilde{w}_k^+(x) &=-\frac{\beta_k}{2\pi} \log|x|+C_{14}+\frac{C_{7, \epsilon_0}}{9\delta_0^2}|x|^{-3\delta_0}\\
\tilde{w}_k^-(x) &=-\frac{\beta_k}{2\pi} \log|x|-C_{14}-\frac{1}{2}|\sigma_k x|^2
\end{align*}
for $C_{14}=1+\max \{C_{12}, C_{13}\}$ and $k \gg 1$, recalling (\ref{eq200}), and let
\begin{equation*}
A_k=B_{(\log \sigma_k^{-1})^{1/\delta_0}} \setminus B_{R'_0}.
\end{equation*}
Then (\ref{eq344}) implies
\begin{align*}
-&\Delta  \tilde{w}_k^+(x) =C_{7, \epsilon_0} |x|^{-(2+3\delta_0)} \geq \tilde{f}_k \quad in \hspace{2mm} A_k, \\
& \tilde{w}_k^+ \geq \tilde{w}_k \quad on \hspace{2mm} \partial A_k.
\end{align*}
Next, we have
\begin{align*}
-&\Delta  \tilde{w}_k^-(x) =-\sigma_k  \leq \tilde{f}_k \quad {in} \hspace{2mm} A_k, \\
& \tilde{w}_k^- \leq \tilde{w}_k \quad {on} \hspace{2mm} \partial A_k.
\end{align*}
Since $-\Delta \tilde{w}_k=\tilde{f}_k$ in $A_k$, it follows from the maximum principle that
\begin{equation} \label{eq360}
\tilde{w}_k^- \leq \tilde{w}_k \leq \tilde{w}_k^+ \quad {in} \hspace{2mm} A_k.
\end{equation}
Using
\begin{equation*}
\Big|\frac{1}{2} |\sigma_k x|^2 \Big| \leq C_{15}, \quad x \in B_{R_0 \sigma_k^{-1}}
\end{equation*}
and
\begin{equation*}
\Big| \frac{C_{7, \epsilon_0}}{9\delta_0^2}|x|^{-3\delta_0} \Big| \leq C_{16}, \quad x \in A_k,
\end{equation*}
we obtain
\begin{equation*}
\big|\tilde{w}_k(x)+\frac{\beta_k}{2\pi} \log|x| \big| \leq C_{14}+\max\{C_{15}, C_{16}\}, \quad x \in A_k
\end{equation*}
for $k \gg 1$.\

Properties (\ref{eq356})-(\ref{eq360}) and (\ref{eq204}) imply that
\begin{equation} \label{eq362}
w_k(x) - w_k(0)=-\Big(\frac{\beta_0}{2\pi}+o(1)\Big) \log (1+e^{w_k (0)/2} |x|)+O(1)
\end{equation}
as $k \to \infty$ uniformly in $x \in B_{R_0}$. We complete the proof of Theorem \ref{th102}. \qed

\section{Proof of Theorem \ref{th106}}
First, we prove Proposition \ref{prop102}.\\

\noindent
{\bf Proof of Proposition \ref{prop102}. } If (\ref{eq124}) holds then we get $S=\{ x_0 \}$ and 
\begin{equation*}
\int_{\mathbb R^2} \tilde{f} dx=m(x_0)=\overline\lambda \int_{I_+} \alpha \mathcal P(d\alpha),
\end{equation*}
which is the mass identity. Furthermore, the above mass identity and (\ref{eq240}) imply that
\begin{equation}
\tilde{\psi}(\beta)=1 \quad \mathcal P{\mathchar `-}a.e \hspace{1mm} on \hspace{1mm} I_+, \quad \alpha_{min}=\beta_{inf} \geq \frac{4\pi}{\beta_0}
\end{equation}
by Proposition \ref{pro205} and Lemma \ref{lem208}. By the argument of Proposition \ref{prop300}, it follows that
\begin{equation*}
\lim_{k \to \infty} \int_{\Omega} e^{\alpha v_k} dx=+\infty
\end{equation*}
for any $\alpha \in$ supp$\mathcal P$. This relation implies $s \equiv 0$ in (\ref{eq114}), that is, the residual vanishing occurs, see \cite{key056}, Lemma 4. The inverse is clearly true. \qed

Before proving Theorem \ref{th106}, we need to prepare some facts with the case general $\mathcal P(d\alpha)$.
It holds that
\begin{equation} \label{eq512}
\int_{I_+} \tilde{\psi}(\beta) \mathcal P(d\beta)=\Big(\int_{I_+} \phi_0(\beta)\tilde{\psi}(\beta)\mathcal P(d\beta)\Big)^2,
\end{equation}
where
\begin{equation} \label{eq514}
\phi_0(\beta)=\sqrt{\frac{\overline \lambda}{8\pi}\beta}.
\end{equation}

Let
\begin{align*}
\mathcal L^0(\psi) &=\int_{I_+} \phi_0(\beta)\psi(\beta) \mathcal P(d\beta)\\
\mathcal C_d &=\{\psi  \mid 0\leq \psi \leq 1 \quad \mathcal P {\mathchar`-} a.e. \hspace{1mm} on\hspace{1mm} I_+ \hspace{1mm}and \hspace{1mm}\int_{I_+} \psi(\beta) \mathcal P(d\beta)=d \}
\end{align*}
and $\chi_A$ be the characteristic function of the set A. The following lemma is a variant of the result of \cite{key026}.
\begin{Lem} \label{lem504}
For each $0<d\leq 1$, the value ${\rm sup}_{\psi \in \mathcal C_d} \mathcal L^0(\psi)$ is attained by
\begin{equation}\label{eq516}
\psi_d(\beta)=\chi_{\phi_0>s_d}(\beta)+c_d \chi_{\phi_0=s_d}(\beta)
\end{equation}
with $s_d$ and $c_d$ defined by
\begin{align} \label{eq518}
s_d&=\inf\{t \mid \mathcal P(\{\phi_0>t\}) \leq d\} \nonumber,\\
c_d\mathcal P(\{\phi_0=s_d\})&=d-\mathcal P(\{\phi_0>s_d\}), \quad 0\leq c_d \leq 1.
\end{align}
Furthermore, the maximizer is unique in the sense that $\psi_m=\psi_d$ $\mathcal P {\mathchar`-}$a.e. on $I_+$ for any maximizer $\psi_m \in \mathcal C_d$. 
\end{Lem}
\begin{proof}
Fix $0<d\leq 1$. Given $\psi \in \mathcal C_d$, we compute
\begin{align}
\int_{I_+} \phi_0(\psi_d-\psi)\mathcal P(d\mathbb \beta)&=\int_{\{\phi_0>s_d\}} \phi_0(\psi_d-\psi)\mathcal P(d\mathbb \beta)+s_d \int_{\{\phi_0=s_d\}} (\psi_d-\psi)\mathcal P(d\mathbb \beta)\nonumber \\
&-\int_{\{\phi_0<s_d\}} \phi_0 \psi \mathcal P(d\mathbb \beta)\nonumber \\
&\geq s_d \int_{\{\phi_0>s_d\}} (\psi_d-\psi)\mathcal P(d\mathbb \beta)+s_d \int_{\{\phi_0=s_d\}} (\psi_d-\psi)\mathcal P(d\mathbb \beta) \label{eq520} \\
&-\int_{\{\phi_0<s_d\}} \phi_0 \psi \mathcal P(d\mathbb \beta) \nonumber   \\
&\geq s_d \Big(\int_{\{\phi_0>s_d\}} (\psi_d-\psi)\mathcal P(d\mathbb \beta)+\int_{\{\phi_0=s_d\}} (\psi_d-\psi)\mathcal P(d\mathbb \beta) \label{eq521} \\
&-\int_{\{\phi_0<s_d\}} \psi \mathcal P(d\mathbb \beta)\Big)  \nonumber \\ 
&=s_d \int_{I_+} (\psi_d-\psi)\mathcal P(d\mathbb \beta)=0, \nonumber
\end{align}
which means that $\psi_d$ is the maximizer.\

The equalities hold in (\ref{eq520}) and (\ref{eq521}) if and only if $\psi$ is the maximizer, and so we shall derive the two conditions. The first condition is that $(\phi_0-s_d)(\psi_d-\psi)=0$ $\mathcal P{\mathchar`-}$a.e. on $\{\phi_0>s_d\}$, so that
\begin{equation}\label{eq522}
\psi=\psi_d \quad \mathcal P {\mathchar `-}a.e. \hspace{1mm}on \hspace{1mm}\{\phi_0>s_d\}
\end{equation}
by the monotonicity of $\phi_0$ and $\psi_d \geq \psi$ on $\{\phi_0>s_d\}$. The second one is that $(s_d-\phi_0)\psi=0$ $\mathcal P$-a.e. on $\{\phi_0<s_d\}$, or
\begin{equation}\label{eq524}
\psi=0 \quad \mathcal P{\mathchar `-}a.e. \hspace{1mm}on \hspace{1mm}\{\phi_0<s_d\}
\end{equation}
by the monotonicity of $\phi_0$ and $\psi \geq 0$. The uniqueness follows from (\ref{eq522}) and (\ref{eq524}) and $\psi_d$, $\psi \in \mathcal C_d$. 
\end{proof}
Let $d \in (0, 1]$ such that $\tilde{\psi} \in \mathcal C_d$ and from (\ref{eq512}), it holds that
\begin{equation*}
d=\int_{I_+} \tilde{\psi}(\beta) \mathcal P(d\beta)=\Big(\int_{I_+} \phi_0(\beta)\tilde{\psi}(\beta)\mathcal P(d\beta)\Big)^2.
\end{equation*}
Lemma \ref{lem212} and (\ref{eq514}) yield
\begin{equation} \label{eq526}
d=\mathcal P(\{\phi_0>s_d\})+c_d \mathcal P(\{\phi_0=s_d\})\leq \Bigg(\sqrt{\frac{\overline\lambda}{8\pi}}\int_{I_+} \psi_d(\beta)\beta\mathcal P(d\beta) \Bigg)^2
\end{equation}
for $\psi_d=\psi_d(\beta)$ defined by (\ref{eq516}) and (\ref{eq518}). By the monotonicity of $\phi_0=\phi_0(\beta)$, there exists the unique element $\beta_d \in I_+$ such that
\begin{equation*}
\phi_0(\beta_d)=s_d,
\end{equation*}
and then (\ref{eq526}) leads
\begin{equation} \label{eq528}
d=\mathcal P(\beta_d, 1]+c_d \mathcal P(\{\beta_d\}) \leq \frac{\overline\lambda}{8\pi} \Bigg(\int_{(\beta_d, 1]} \beta \mathcal P(d\beta)+c_d \beta_d \mathcal P(\{\beta_d\})\Bigg)^2.
\end{equation}
Here we introduce
\begin{equation} \label{eq530}
H(\tau)=\mathcal P(\beta_d, 1]+\tau \mathcal P(\{\beta_d\})-\frac{\overline\lambda}{8\pi}\Bigg(\int_{(\beta_d, 1]} \beta \mathcal P(d\beta)+\tau \beta_d \mathcal P(\{\beta_d\})\Bigg)^2.
\end{equation}
It follows from (\ref{eq126}) that
\begin{equation} \label{eq532}
H(0) \geq 0, \quad H(1) \geq0.
\end{equation}
{
\begin{Rem}\label{rem400}
Here we use the property of $\overline \lambda$ for (\ref{eq532}).
\end{Rem}
}
Moreover, we have either $c_d=0$ or $c_d=1$ if $\mathcal P(\{\beta_d\})>0$. In fact, since
\begin{equation*}
H''(\tau)=-\frac{\overline\lambda}{4\pi}\Bigg(\beta_d \mathcal P(\{\beta_d\})\Bigg)^2<0
\end{equation*}
by $\mathcal P(\{\beta_d\})>0$, it holds that $H(\tau)>0$ for $0<\tau<1$ by (\ref{eq532}). On the other hand, $H(c_d) \leq 0$ by (\ref{eq530}).\

We now claim
\begin{equation}\label{eq534}
\tilde{\psi}=\psi_d=\chi_{I_d} \quad \mathcal P{\mathchar `-} a.e. \hspace{1mm}on \hspace{1mm} I_+
\end{equation}
where
\begin{equation*}
  I_{d} = \begin{cases}
    [\beta_{d}, 1] &{\rm if} \hspace{1mm}\mathcal P(\{\beta_{d}\})>0 \hspace{1mm}{\rm and} \hspace{1mm}c_d=1,\\
    (\beta_{d}, 1] &{\rm if} \hspace{1mm}\mathcal P(\{\beta_{d}\})=0 \hspace{1mm}{\rm or} \hspace{2mm} \mathcal P(\{\beta_{d}\})>0 \hspace{1mm}{\rm and} \hspace{1mm}c_d=0.
  \end{cases}
\end{equation*}
First, we assume that $\mathcal P(\{\beta_{d}\})=0$. Then, $H(\tau)=H(0)$ for $\tau \in [0, 1]$. In this case, the equality holds in (\ref{eq530}) by (\ref{eq532}), and thus
\begin{equation*}
d=\Bigg(\int_{I_+} \phi_0(\beta)\psi_d(\beta) \mathcal P(d\beta)\Bigg)^2=\Bigg(\int_{I_+} \phi_0(\beta)\tilde{\psi}(\beta) \mathcal P(d\beta)\Bigg)^2,
\end{equation*}
which means $\tilde{\psi}=\psi_d$ $\mathcal P$-a.e.on $I_+$ by the uniqueness of Lemma \ref{lem504}. Note that the integrations are non-negative. It is clear that $\psi_d=\chi_{I_d}$ $\mathcal P$-a.e. on $I_+$. Next we assume that $\mathcal P(\{\beta_d \})>0$. Then we use (\ref{eq530}) and (\ref{eq532}) to obtain $H(c_d)=0$, which again implies that the equality holds in (\ref{eq530}), and hence
\begin{equation*}
  \tilde{\psi}=\psi_d = \begin{cases}
    \chi_{[\beta_d, 1]} &{\rm if} \hspace{2mm} c_d=1,\\
    \chi_{(\beta_d, 1]} &{\rm if} \hspace{2mm} c_d=0.
  \end{cases}
\end{equation*}
The claim (\ref{eq534}) is established.\

Here we divide two cases as $\beta_d > \beta_{inf}$ and $\beta_d \leq \beta_{inf}$.\\
First, we consider the case $\beta_d > \beta_{inf}$. Then, we have,
\begin{equation*}
\mathcal P(I_{inf}\setminus I_{d})=0.
\end{equation*}
Indeed, assume $\mathcal P(I_{inf} \setminus I_d)>0$. Then,
\begin{equation}\label{eq535}
\tilde{\psi}(\beta)=0 \quad {\rm for} \hspace{1mm} \mathcal P{\mathchar `-} a.e. \hspace{1mm} \beta \in I_{inf}\setminus I_d
\end{equation}
by (\ref{eq534}). On the other hand, $\tilde{\psi}(\beta)>0$ for any $\beta \in I_{inf}\setminus I_d$ by the definition of $I_{inf}$ and $\tilde{\psi}$, and by the convergence (\ref{eq210}), which contradicts (\ref{eq535}).\

In the case of $\beta_d \leq \beta_{inf}$, we obtain the following result:
\begin{Prop}\label{prop502}
Suppose $\beta_d \leq \beta_{inf}$ then it holds that
\begin{equation*}
\tilde{\psi}(\beta)=\chi_{I_{inf}}(\beta) \quad \mathcal P{\mathchar `-}a.e \hspace{1mm} \beta,
\end{equation*}
where 
\begin{equation*}
  I_{inf} = \begin{cases}
    [\beta_{inf}, 1] &{\rm if} \hspace{2mm} \beta_{inf}\in  \mathcal B,\\
    (\beta_{inf}, 1] &{\rm if} \hspace{2mm} \beta_{inf}\not\in  \mathcal B.
  \end{cases}
\end{equation*}
\end{Prop}

\begin{proof}
There are the following five possibilities:
\begin{align*}
&{\rm(i)} \hspace{4mm} \beta_{d}<\beta_{inf},\\
&{\rm(ii)} \hspace{3mm} \beta_{d}=\beta_{inf}, \hspace{2mm} I_d=(\beta_d, 1] \hspace{2mm} and \hspace{2mm} \beta_{inf} \in I_{inf},\\
&{\rm(iii)} \hspace{2mm} \beta_{d}=\beta_{inf}, \hspace{2.2mm} I_d=[\beta_d, 1] \hspace{2mm} and \hspace{2mm} \beta_{inf} \not\in I_{inf},\\
&{\rm(iv)} \hspace{2mm} \beta_{d}=\beta_{inf}, \hspace{2.2mm} I_d=[\beta_d, 1] \hspace{2mm} and \hspace{2mm} \beta_{inf} \in I_{inf},\\
&{\rm(v)} \hspace{3mm} \beta_{d}=\beta_{inf}, \hspace{2mm} I_d=(\beta_d, 1] \hspace{2mm} and \hspace{2mm} \beta_{inf} \not\in I_{inf}.
\end{align*}

The result is clearly true for the cases (iv)-(v), and thus it suffices to prove $\mathcal P(I_d\setminus I_{inf})=0$, and $\mathcal P(\{\beta_d \})=\mathcal P(\{\beta_{inf}\})=0$ for the cases (i) and (ii), (iii), respectively.\

(i) Assume $\mathcal P(I_d \setminus I_{inf})>0$. Then
\begin{equation}\label{eq536}
\tilde{\psi}(\beta)=0 \quad {\rm for} \hspace{1mm} \beta \in I_d \setminus I_{inf}
\end{equation}
by the definitions of $I_{inf}$ and $\tilde{\psi}$. Note that $\tilde{w}_{k, \beta} \to -\infty$ locally uniformly in $\mathbb R^2$ for $\beta \in I_d \setminus I_{inf}$. On the other hand, $\tilde{\psi}(\beta)=1$ for some $\beta \in I_d \setminus I_{inf}$ by (\ref{eq534}), which contradicts (\ref{eq536}).\

(ii) If $\mathcal P(\{\beta_d\})=\mathcal P(\{\beta_{inf}\})>0$ then $\tilde{\psi}(\beta_d)=\tilde{\psi}(\beta_{inf})=0$ by (\ref{eq534}) and $I_d=(\beta_d, 1]$. On the other hand, $\tilde{\psi}(\beta_d)=\tilde{\psi}(\beta_{inf})>0$ by $\beta_{inf} \in I_{inf}$ as shown for the case $\beta_d >\beta_{inf}$ above, a contradiction.\

(iii) If $\mathcal P(\{\beta_d\})=\mathcal P(\{\beta_{inf}\})>0$ then $\tilde{\psi}(\beta_d)=\tilde{\psi}(\beta_{inf})=1$ by (\ref{eq534}) and $I_d=[\beta_d, 1]$. On the other hand, $\tilde{\psi}(\beta_d)=\tilde{\psi}(\beta_{inf})=0$ by $\beta_{inf} \not\in I_{inf}$ as shown for the case (i) above, a contradiction.
\end{proof}

{\bf Proof of Theorem \ref{th106}.} Let $\mathcal P(d\alpha)$ be as in (\ref{eq136}). First, we consider the case $\sqrt{\tau}/(1+\sqrt{\tau}) <\gamma <1$. Now, we divide this proof as two cases:
\begin{equation*}
\beta_d >\beta_{inf} \quad and \quad \beta_d \leq \beta_{inf}.
\end{equation*}
First, we consider the case $\beta_d >\beta_{inf}$. In this case, by Lemma \ref{lem212} and (\ref{eq534}), we have,
\begin{equation*}
\tau=\frac{\tau^2}{(\tau+(1-\tau)\gamma)^2},
\end{equation*}
that is, $\gamma=\sqrt{\tau}/(1+\sqrt{\tau})$ which is a contradiction to $\gamma > \sqrt{\tau}/(1+\sqrt{\tau})$. Therefore we just consider the case $\beta_d \leq \beta_{inf}$.
Note that we have $\gamma=\beta_{inf}$ and $\gamma \in \mathcal B$. Indeed, if $\gamma<\beta_{inf}$ or $\gamma \not\in \mathcal B$ holds then we can lead a contradiction by the same argument of the case $\beta_d >\beta_{inf}$ thanks to Proposition \ref{prop502}.\\
Since  $\gamma=\beta_{inf}$ and $\gamma \in \mathcal B$ holds, by Proposition \ref{prop502}  we have
\begin{equation} \label{eq538}
\tilde{\psi}(\beta)=\chi_{\{\gamma, 1\}}(\beta)\quad \mathcal P{\mathchar `-}a.e \hspace{1mm} \beta.
\end{equation}
By (\ref{eq538}) and Proposition \ref{pro203}, we obtain the following identity:
\begin{equation*}
\int_{\mathbb R^2}\tilde{f}dy=\overline \lambda \int_{I_+} \beta \chi_{\{\gamma, 1\}}(\beta) \mathcal P(d\beta)=\overline \lambda(\tau+(1-\tau)\gamma).
\end{equation*} 
By Proposition \ref{prop102}, the above identity implies the estimate (\ref{eq122}).\

Next, in the case $0< \gamma < \sqrt{\tau}/(1+\sqrt{\tau})$ we suppose the identity (\ref{eq124}) holds. From this assumption and Proposition \ref{prop102}, the residual vanishing occurs and we have
\begin{equation*}
\overline \lambda=\frac{8\pi}{\big( \tau+(1-\tau)\gamma \big)^2},
\end{equation*}
which is a contradiction to $\overline \lambda=8\pi / \tau< 8\pi /( \tau+(1-\tau)\gamma)^2$. \qed

{
\begin{Rem}\label{rem402}
In the case of  $0< \gamma < \sqrt{\tau}/(1+\sqrt{\tau})$, we use the property of $\overline \lambda$ again for the contradiction.
\end{Rem}
}

\section*{Acknowledgements}
The authors would like to thank Associated Professor Ryo Takahashi from Nara University of Education for valuable comments. This work is supported by JSPS Grant-in-Aid for Scientific Research (A) 26247013.

\appendix
\def\thesection{Appendix}
\section{Proof of Lemma \ref{lem203}}
Given $K>0$, we put
\begin{align*}
I_1(x)&=\int_{D_1} \frac{\log|x-y| - \log(1+|y|) - \log|x|}{\log|x|}f(y)dy,  \\
I_{2, K}(x)&=\int_{D_{2, K}} \frac{\log|x-y| - \log(1+|y|) - \log|x|}{\log|x|}f(y)dy, \\
I_{3, K}(x)&=\int_{D_{3, K}} \frac{\log|x-y| - \log(1+|y|) - \log|x|}{\log|x|}f(y)dy, \\
\end{align*}
where,
\begin{align*}
D_1&=D_1(x) \equiv \{y \in \mathbb R^2 \mid |y-x|<1 \},\\
D_{2, K}&=D_{2, K}(x) \equiv \{y \in \mathbb R^2 \mid |y-x|>1, |y| \leq K \},\\
D_{3, K}&=D_{3, K}(x) \equiv \{y \in \mathbb R^2 \mid |y-x|>1, |y| > K \}.
\end{align*}
Then it holds that
\begin{equation*}
\frac{z(x)}{\log |x|}-\frac{\beta_0}{2\pi}=\frac{1}{2\pi}(I_1(x)+I_{2, K}(x)+I_{3, K}(x)).
\end{equation*}
We have only to show that each $\epsilon>0$ admits $K_{\epsilon}$ and $L_{\epsilon}$ such that
\begin{equation} \label{A001}
|I_1(x)|+|I_{2, K}(x)|+|I_{3, K}| \le \epsilon
\end{equation}
for all $x \in \mathbb R^2 \setminus B_{L_{\epsilon}}$.\

Since
\begin{equation*}
\frac{\log(1+|y|)+\log|x|}{\log|x|} \leq \frac{\log(2+|x|)+\log|x|}{\log|x|}\leq 3, \quad x \in \mathbb R^2 \setminus B_2, \quad y \in D_1(x),
\end{equation*}
we have,
\begin{align} \label{A003}
|I_1(x)| & \leq 3 \int_{D_1} f(y) dy+ \frac{1}{\log |x|}  \int_{D_1} f(y) \log|x-y| dy \nonumber \\
&\leq 3 \int_{D_1} f(y) dy+ \frac{\| f \|_{\infty}}{\log |x|}  \int_{D_1} f(y) \log|y| dy \to 0
\end{align}
uniformly as $|x| \to +\infty$, recalling $f \in L^1 \cap L^{\infty} (\mathbb R^2)$.\

Next, we have
\begin{equation*}
\Bigg| \frac{\log|x-y| - \log(1+|y|) - \log|x|}{\log|x|} \Bigg| \leq \frac{1}{\log|x|} \Bigg\{ \log(1+K)+\log\frac{|x-y|}{|x|} \Bigg\}
\end{equation*}
for $x \in \mathbb R^2 \setminus B_2$ and $y \in D_{2, K}(x)$, and thus
\begin{equation} \label{A005}
| I_{2, K}(x) | \leq \int_{D_{2, K}(x)} \Bigg\{ \log(1+K)+\log\frac{|x-y|}{|x|} \Bigg\}f(y)dy
\end{equation}
for $x \in \mathbb R^2 \setminus B_2$. From
\begin{equation*}
\frac{1}{2+|x|} \leq \frac{|x-y|}{1+|y|} \leq 1+|x|, \quad x \in \mathbb R^2, \hspace{2mm}|y-x| \geq 1,
\end{equation*}
we derive
\begin{equation*}
\Bigg| \frac{\log|x-y| - \log(1+|y|) - \log|x|}{\log|x|} \Bigg| \leq 3, \quad x \in \mathbb R^2 \setminus B_2 , \hspace{2mm}|y-x| \geq 1,
\end{equation*}
to obtain
\begin{equation} \label{A007}
| I_{3, K}(x) | \leq 3\int_{D_{3, K}(x)} f(y) dy \leq 3\int_{\mathbb R^1 \setminus B_K} f(y) dy
\end{equation}
for $x \in \mathbb R^2 \setminus B_2$.

Recalling $0\leq f \in L^1(\mathbb R^2)$, let $\epsilon_0 >0$ be given. From (\ref{A007}), there exists $K_0>0$ such that
\begin{equation*}
|I_{3, K}(x) | \leq \epsilon_0
\end{equation*}
for all $K \geq K_0$ and $x \in \mathbb R^2 \setminus B_2$. Next, by (\ref{A003}) any $K>0$ admits $L_K>0$ such that
\begin{equation*}
|I_{2, K}(x) | \leq \epsilon_0
\end{equation*}
for all $x \in \mathbb R^2 \setminus B_{L_{K_0}}$, and therefore
\begin{equation}  \label{A009}
|I_{2, K_0}(x) | +|I_{3, K_0}(x) |\leq 2\epsilon_0
\end{equation}
for all $x \in \mathbb R^2 \setminus B_{L_{K_0}}$.\

Thus we obtain (\ref{A001}) by (\ref{A003}) and (\ref{A009}). \qed

\end{document}